\documentclass[11pt]{amsart}
\allowdisplaybreaks
\usepackage{amsrefs} 
\usepackage{fancyhdr}
\usepackage[titletoc]{appendix}

\usepackage{amsmath,amssymb,amsthm}
\usepackage{bm}
\usepackage[dvipsnames]{xcolor}
\usepackage{etoolbox,adjustbox}
\usepackage{abstract}

\usepackage[overload]{textcase}
\usepackage{multirow}
\usepackage{graphics,graphicx}
\usepackage{float,verbatim,array}
\usepackage{todonotes}
\usepackage{xfrac,enumitem,soul}
\usepackage{tikz,tikz-cd}
\usepackage[breaklinks=true]{hyperref}
\usepackage[top=1in, bottom=1in, left=1in, right=1in]{geometry}
\usepackage{color,xcolor}
\colorlet{BLUE}{blue}
\colorlet{RED}{red}
\colorlet{GRAY}{gray}
\colorlet{BROWN}{brown}
\definecolor{OliveGreen}{rgb}{0,0.6,0}

\theoremstyle{definition}
\newtheorem{dfn}{Definition}
\newtheorem{note}{Note}
\newtheorem{question}{Question}
\newtheorem{fact}{Fact}

\theoremstyle{remark}
\newtheorem*{rmk}{Remark}

\theoremstyle{plain}
\newtheorem{thm}{Theorem}[section]
\newtheorem{lmm}[thm]{Lemma}

\title[Modulo arithmetic of function spaces: Subset hyperspaces as quotients of function spaces]{Modulo arithmetic of function spaces: Subset hyperspaces as quotients of function spaces}
\author{Earnest Akofor}
\address{\textnormal{Department of Mathematics and Computer Science, 
         Faculty of Science, University of Bamenda, 
         PO Box 39 Bambili, NW Region, Cameroon}}
\email{eakofor@gmail.com}

\subjclass[2020]{Primary 54B20 54C35; Secondary 54B15 54E35 54A10 54C50}
\keywords{Subset hyperspace, function space, quotient topology, swrc-topology.}

\begin{document}

\begingroup
\def\uppercasenonmath#1{} 
\let\MakeUppercase\relax 
\maketitle

\begin{abstract}    
\vspace{0.2cm}

\noindent Let $X$ be a (topological) space and $Cl(X)$ the collection of nonempty closed subsets of $X$. Given a topology on $Cl(X)$, making $Cl(X)$ a space, a \emph{(subset) hyperspace} of $X$ is a subspace $\mathcal{J}\subset Cl(X)$ with an embedding $X\hookrightarrow\mathcal{J}$, $x\mapsto\{x\}$. In this note, we characterize certain hyperspaces $\mathcal{J}\subset Cl(X)$ as explicit quotient spaces of function spaces $\mathcal{F}\subset X^Y$ and discuss metrization of associated compact-subset hyperspaces in this setting. In particular, we find that any hyperspace topology containing the Vietoris topology is a quotient of a function space topology containing the topology of pointwise convergence.
\end{abstract}
\let\bforigdefault\bfdefault
\addtocontents{toc}{\let\string\bfdefault\string\mddefault}
\tableofcontents
\section{\textnormal{\bf Introduction}}\label{Intro}
\noindent The phrase ``modulo arithmetic'' in the title of this paper is based on the understanding that the operation of taking quotients of (algebraic or geometric) structures by substructures to obtain new structures might be seen as generalizing the modulo arithmetic of integers, with quotient rings being the closest generalizations of the ring $\mathbb{Z}_n:=\mathbb{Z}/n\mathbb{Z}$ of modulo $n$ integers. Our situation in particular admits the simplified formula
\[
\textrm{(subset hyperspace) = (function space) modulo (quotient map).}
\]
Theorem \ref{QuotRepVTop}'s Remark (\ref{SSpConnect4})-(\ref{SSpConnect6}) makes a connection with quotients of groups, rings, and modules.

\vspace{0.2cm}
\begin{center}
{\emph{Preliminary notation and terminology}}
\end{center}
\vspace{0.1cm}

\noindent Given a set $X$ and a collection $\mathcal{B}\subset\mathcal{P}(X)$ of subsets of $X$, the topology on $X$ generated by $\mathcal{B}$, i.e., with $\mathcal{B}$ as a \textbf{subbase}, is denoted by $\langle\mathcal{B}\rangle$. Given two (topological) spaces $X_1$ and $X_2$, $X_1\cong X_2$ means $X_1$ and $X_2$ are \textbf{homeomorphic}. Let $X=(X,\tau)$ be a space. A topology $\tau'\subset\mathcal{P}(X)$ on $X$ is \textbf{$\tau$-compatible} if $\tau\subset\tau'$ (where $\tau'$ can therefore inherit certain desirable properties of $\tau$, e.g., if $\tau$ is $T_0$, $T_1$, or $T_2$ respectively, then so is $\tau'$). If $A\subset X$, the closure of $A$ in $X$ is denoted by $cl_X(A)$, or by $\overline{A}$ if the underlying space $X$ is understood. We say that $A\subset X$ is \textbf{precompact in $X$} if $cl_X(A)$ is compact in $X$. Let $Z$ be a metric space. The \textbf{completion} of $Z$ is (up to isometry) the complete metric space $\widetilde{Z}$ containing $Z$ as a dense subset. We say $A\subset Z$ is \textbf{totally bounded} if $A$ is precompact in $\widetilde{Z}$. It is easy to verify that a set $A\subset Z$ is totally bounded in the usual sense if and only if $A$ is precompact in $\widetilde{Z}$. It might be worth noting that in the context of \cite{akofor2024}, ``precompact'' is automatically equivalent to ``totally bounded'', as explained by \cite[Notation 1]{akofor2024}.

Let $\mathcal{H}(X)$ denote the class of homeomorphisms of $X$ and let $\mathcal{C}\subset\mathcal{H}(X)$. By \textbf{geometry} of $X$ (resp., \textbf{$\mathcal{C}$-geometry} of $X$) we mean the study of one or more properties of $X$ that are invariant under specific homeomorphisms of $X$ (resp., homeomorphisms of $X$ in $\mathcal{C}$), where the invariant properties are accordingly called \textbf{geometric properties} (resp., \textbf{$\mathcal{C}$-geometric properties}) of $X$. By \textbf{metric geometry} we mean geometry that employs metrics. With $\mathcal{P}^\ast(X)$ denoting the set of nonempty subsets of $X$, we would like to study the simplest kinds of topologies on $\mathcal{J}\subset\mathcal{P}^\ast(X)$ using better understood topologies on $\mathcal{F}\subset X^Y$ (for sets $Y$). Since a $T_1$ space $X$ can be seen as a subset of $\mathcal{P}^\ast(X)$ in a natural way (through the inclusion $X\hookrightarrow\mathcal{P}^\ast(X),~x\mapsto\{x\}$), we often consider those topologies on $\mathcal{J}\supset\textrm{Singlt}(X):=\big\{\{x\}:x\in X\big\}$ that can be seen as extensions of the topology of $X$. Such topologies are called \textbf{hypertopologies} of $X$, and the associated spaces $\mathcal{J}$, $\textrm{Singlt}(X)\subset\mathcal{J}\subset\mathcal{P}^\ast(X)$, are called \textbf{(subset) hyperspaces} of $X$. A \textbf{closed-subset hyperspace}, \textbf{compact-subset hyperspace}, or \textbf{bounded-subset hyperspace} of $X$ is a hyperspace consisting respectively of closed subsets, compact subsets, or bounded subsets of $X$.

Let $X$ be a $T_1$ space, $Cl(X)\subset\mathcal{P}^\ast(X)$ the set of nonempty \textbf{closed subsets} of $X$ as a hyperspace whose topology is comparable to the \textbf{Vietoris topology} $\tau_v$, $K(X)\subset Cl(X)$ the subspace consisting of all nonempty \textbf{compact subsets} of $X$, and $FS_n(X):=\{A\in K(X):1\leq|A|\leq n\}$ the subspace consisting of all nonempty \textbf{finite subsets} of $X$ of cardinality at most $n$. When $X$ is a metric space, we further let $BCl(X)\subset Cl(X)$ denote the subspace consisting of all nonempty \textbf{bounded closed subsets} of $X$. Let $Y$ be a set. On relevant sets of functions $\mathcal{F}\subset X^Y$, besides the standard \textbf{product topology} (or \textbf{topology of pointwise convergence}) $\tau_p$ and related topologies, we will introduce a preferred topology $\tau_{\pi}$ on $\mathcal{F}\subset X^Y$ (Definition \ref{StdTopIndSSP}) in order to discuss metrization of compact-subset hyperspaces (Theorems \ref{ScpTopLmm0} and \ref{HausdMetriz1}). Since a closed set $C\in Cl(X)$ can be seen as the closure  of the image $f(Y)$ of some function $f\in X^Y$, the image-closure assignment (named \textbf{unordering map} in Definition \ref{ProdDescDfn}) given by
\begin{align}
\label{UnordMapEq}q:(\mathcal{F},\tau_{\pi})\rightarrow (Cl_Y(X),\tau_{\pi q}),~ f\mapsto cl_X(f(Y))
\end{align}
induces a quotient topology $\tau_{\pi q}$, the \textbf{$\tau_{\pi}$-quotient topology} (\textbf{footnote}\footnote{
That is, $\tau_{\pi q}:=\sup\big\{\tau~|~q:(\mathcal{F},\tau_{\pi})\rightarrow (Cl_Y(X),\tau)~\textrm{is surjective and continuous}\big\}$ is the largest topology on $Cl_Y(X)$ with respect to which the map $q:(\mathcal{F},\tau_{\pi})\rightarrow Cl_Y(X)$ is surjective and continuous.
}), on the \textbf{$Y$-indexed closed subsets}
\begin{align}
\label{IndCloSubEq}Cl_Y(X):=q(\mathcal{F}).
\end{align}

\vspace{0.2cm}
\begin{center}
{\emph{Motivation}}
\end{center}
\vspace{0.1cm}

\noindent First, we would like to have a straightforward description (by means of a natural quotient map) of the relation between basic function space topologies (e.g., the topologies of pointwise convergence, compact convergence, uniform convergence) and the Vietoris topology of subset hyperspaces. We expect this to facilitate the use of function spaces to study subset hyperspaces and vice versa.

Second, for a metric space $X$, we seek a suitable geometric framework for answering \cite[Question 5.1]{akofor2024} concerning characterization/representation of Lipschitz paths in $BCl(X)$ in terms of Lipschitz paths in $X$ or $\widetilde{X}$ (as indicated in Question \ref{GHDQuestion00}). Due to the key role of Hausdorff distance in the results of \cite{akofor2024},
the desired geometric framework appears to require a deeper understanding of the relationship between certain basic function space topologies and the metrization of associated hypertopologies by the Hausdorff distance $d_H$ (Equation (\ref{HausDistDef})).

Our work is therefore primarily motivated by the need to use \emph{metrizability} of certain hyperspaces $\mathcal{J}\subset K(X)$, viewed as quotients of function spaces $\mathcal{F}\subset X^Y$ (suitably topologized), to strengthen our \emph{understanding}/\emph{interpretation} of $d_H$ as a key metric on $\mathcal{J}$. Our main results (Theorems \ref{ScpTopThm}, \ref{ScpTopLmm0}, \ref{HausdMetriz1}, and \ref{ScpTopThmPrv}) are based on the observation (via Lemma \ref{QuotVietCont}, Lemma \ref{QuotVietContU}, and Theorem \ref{QuotRepVTop}) that basic hyperspace topologies (with reasonable separation properties) admit a natural description as quotients of basic function space topologies (with reasonable separation properties).

\vspace{0.2cm}
\begin{center}
{\emph{Summary and highlight of main results}}
\end{center}
\vspace{0.1cm}

\noindent To describe our main results, consider the following question.

\vspace{0.2cm}
\begin{enumerate}[leftmargin=0.0cm]
\item[] \textbf{Question A}: Let $X$ be a space, $Y$ a set, and $\mathcal{F}\subset X^Y$ a subset satisfying $q(\mathcal{F})=q(X^Y)$. Define $Cl_Y(X):=q(X^Y)$ and $K_Y(X):=Cl_Y(X)\cap K(X)$, where $q$ is the map given by
\[
q:\mathcal{F}\rightarrow Cl_Y(X),~~f\mapsto cl_X(f(Y)).
\]
\end{enumerate}
\begin{enumerate}
\item\label{QstAPt1} Can we choose and characterize a topology $\tau_\pi\supset \tau_p$ on $\mathcal{F}$ such that the following hold?
    \begin{enumerate}
     \item\label{QstAPt1a} The map $q:(\mathcal{F},\tau_\pi)\rightarrow Cl_Y(X)$ induces a quotient topology $\tau_{\pi q}\supset \tau_v$ on $Cl_Y(X)$.
     \item\label{QstAPt1b} The compact-subset hyperspace $(K_Y(X),\tau_{\pi q})$ is metrizable whenever $X$ is metrizable.
    \end{enumerate}
\item\label{QstAPt2} If $\tau\supset\tau_v$ is a topology on $Cl_Y(X)$, does there exist a topology $\widetilde{\tau}\supset \tau_p$ on $\mathcal{F}$ such that $\tau\supset\widetilde{\tau}_q\supset\tau_v$ (where $\widetilde{\tau}_q$ is the quotient topology induced on $Cl_Y(X)$ by $\widetilde{\tau}$ via the map $q$)?
\end{enumerate}
\vspace{0.2cm}

In Section \ref{SHQFS}, we describe a class of hyperspaces as quotients of function spaces, noting that some related work has been considered in \cite{bartsch2014} (see Fact \ref{Bch14Fact} below). Next in Section \ref{QRMCSH}, we discuss the concrete realization of certain preferred function space topologies and metrization of compact-subset hyperspaces, and (under some regularity constraints on $\mathcal{F}$) give a positive answer to Question A (\ref{QstAPt1a})$\&$(\ref{QstAPt1b}) in Theorems \ref{ScpTopThm} $\&$ \ref{HausdMetriz1} respectively. In Section \ref{QLHT}, we discuss the realization of $\tau_v$-compatible hyperspace topologies as quotients of $\tau_p$-compatible function space topologies, and give a positive answer to Question A(\ref{QstAPt2}) in Theorem \ref{ScpTopThmPrv}.

\begin{fact}[Definition \ref{SetOpTop}'s Remark (\ref{TpiChRmk2})]\label{Bch14Fact}
For any $T_3$-space (i.e., regular Hausdorff space) $X$, according to \cite[Theorem 2.4]{bartsch2014}, there always exists a compact space $Y$ such that we have a quotient map
    \[
    q:C(Y,X)\subset (X^Y,\tau_{\textrm{co}})\rightarrow (K(X),\tau_v),
    \]
    where $C(Y,X):=\{\textrm{continuous}~f\in X^Y\}$ and $\tau_{co}$ is the \textbf{compact-open topology} (Definition \ref{SetOpTop}).
\end{fact}

Fact \ref{Bch14Fact} may be seen as a special case of Theorem \ref{ScpTopThmPrv} (on the existence of a $\tau_p$-compatible topology on $\mathcal{F}\subset X^Y$ with a $\tau_v$-compatible quotient). We conclude in Section \ref{DiscQuest} with some interesting questions.

Throughout, we appeal to intuition by preferably employing sequences and nets (instead of open sets) in our results and proofs whenever this seems convenient.

\section{\textnormal{\bf Preliminary remarks}}\label{PrelRmks}
\noindent Subset hyperspaces have been studied as function spaces in \cite{Render1993,McCoy1998,DolMyn2010}, and as quotient spaces of function spaces in \cite{bartsch2014,bartsch2004,bartsch2002}. These works mainly characterize and compare various geometric properties of those restrictions and extensions of the Vietoris topology that arise through classical (operations like embedding and compactification around) function space topologies. Our discussion is more focused towards explicitly constructing certain generalizations of the pointwise convergence topology $\tau_p$ (of function spaces) whose quotients agree with the metrizable Vietoris topology $\tau_v$ of compact-subset hyperspaces.

One of our main results, namely, metrization of compact-subset hyperspaces in Theorem \ref{HausdMetriz1}, is well-known for standard hyperspaces with the Vietoris topology (see Lemma \ref{VietTopMetr}, \cite[Theorems 2.4, 3.1, 3.2, 3.4]{IllanNadl1999}, \cite[Theorem 4.9.13]{michael1951}, and \cite{DolMyn2010,Bart2014,McCoy1998,RodRom2002,LeviEtal1993,Render1993}). Our goal here is simply to present a specialized review that emphasizes the description of $\tau_v$-compatible subset hyperspaces as quotients of $\tau_p$-compatible function spaces (where the role originally played by the Vietoris topology $\tau_v$ is now played by $\tau_v$-compatible quotient topologies induced by relevant function space topologies). The discussion may be viewed as an extension of preliminary discussions for the case of finite-subset hyperspaces in \cite[Chapter 1, especially around Definition 1.0.1 and Proposition 1.2.2]{akofor2020}.

\begin{dfn}[{Hausdorff distance}]
Let $X$ be a metric space, $x\in X$, $A\subset X$, $\varepsilon>0$,
\begin{align*}
&\textstyle d(x,A):=\inf\limits_{a\in A}d(x,a),\\
&\textstyle N_\varepsilon(A):=\{x\in X:d(x,A)<\varepsilon\},~\textrm{and}~\\
&\overline{N}_\varepsilon(A)=A_\varepsilon:=\{x\in X:d(x,A)\leq\varepsilon\}.
\end{align*}
If $A,B\in BCl(X)$, the \textbf{Hausdorff distance} between $A$ and $B$ is
\begin{align}
\label{HausDistDef} d_H(A,B)&\textstyle:=\max\{\sup_{a\in A}d(a,B),\sup_{b\in B}d(b,A)\}\\
&=\inf\left\{r>0:A\cup B\subset \overline{N}_r(A)\cap \overline{N}_r(B)\right\}\nonumber\\
&=\textstyle\sup_{x\in A\cup B}|d(x,A)-d(x,B)|\nonumber\\
&\textstyle=\sup_{x\in X}|d(x,A)-d(x,B)|.\nonumber
\end{align}
\end{dfn}

\begin{dfn}[{Saturated set}]\label{QuotSpDfn}
Let $g:A\rightarrow B$ be a map. A subset $S\subset A$ is \textbf{$g$-saturated} if $g^{-1}(g(s))\subset S$ for every $s\in S$ (i.e., $g^{-1}(g(S))=S$, or equivalently, $S=g^{-1}(T)$ for some $T\subset B$).
\end{dfn}

Recall that a quotient map $f:(U,\tau_U)\rightarrow (V,\tau_V)$ makes $V$ a quotient (or quotient space) of $U$, in which case we will call $\tau_V$ the \textbf{$f$-quotient topology} induced on $V$ by $\tau_U$, written as
\[
\tau_V=\tau_{Uf},
\]
which is a prelude to the notation introduced in Item \ref{HPSpDfnIt4} of Definition \ref{ProdDescDfn}.

\begin{rmk}\label{QuotTopCaut}
Let $q:U\rightarrow V$ be a quotient map.
\begin{enumerate}[leftmargin=0.8cm]
\item\label{QuotTopCautIt1} Given $I\subset U$, the restriction $q|_I:I\rightarrow q(I)$ need not be a quotient map with respect to the subspace topologies. Indeed, if $P$ is open in $V$, i.e., $E:=P\cap q(I)$ is open in $q(I)$, then $q|_I^{-1}(E):=q^{-1}(E)\cap I=q^{-1}(P)\cap I$ (which is open in $I$), but $q^{-1}(F)\cap I=q^{-1}(P)\cap I$ (for some $F\subset V$) does not always imply that $F\cap q(I)$ is open in $q(I)$.
\item If $O\subset U$ (resp., $C\subset U$) is a $q$-saturated open (resp., closed) set, then $q|_O:O\rightarrow q(O)$ (resp., $q|_C:C\rightarrow q(C)$) is a quotient map with respect to the subspace topologies. Indeed, in Part (\ref{QuotTopCautIt1}) above with $I$ replaced by $O$, $q^{-1}(F)\cap O=q^{-1}(P)\cap O$ ($\iff$ $F\cap q(O)=P\cap q(O)$) is open in $U$ if and only if $F\cap q(O)$ is open in $q(O)$.
\item\label{QuotTopCautIt3} In particular, if $O\subset V$ (resp., $C\subset V$) is an open (resp., closed) set, then $q^{-1}(O)\subset U$ (resp., $q^{-1}(C)\subset U$) is a $q$-saturated open (resp. closed) set, giving a quotient map
    \[
    q|_{q^{-1}(O)}:q^{-1}(O)\rightarrow O,~~(\textrm{resp.,}~~q|_{q^{-1}(C)}:q^{-1}(C)\rightarrow C).
    \]
\end{enumerate}

Due to (\ref{QuotTopCautIt1}) above, a quotient map $q:(X^Y,\tau)\rightarrow (q(X^Y),\tau')$ need not automatically restrict to a quotient map $q:(\mathcal{I},\tau)\subset(X^Y,\tau)\rightarrow (q(\mathcal{I}),\tau')\subset(q(X^Y),\tau')$. Consequently, we will typically (i) fix a relevant subset $\mathcal{F}\subset X^Y$ satisfying $q(\mathcal{F})=q(X^Y)$ and (ii) directly specify quotient maps
\[
q:(\mathcal{F},\tau)\subset X^Y\rightarrow (q(X^Y),\tau_q),
\]
which need not be restrictions of quotient maps $(X^Y,\tau')\rightarrow (q(X^Y),\tau'_q)$, even if $(\mathcal{F},\tau)\subset(X^Y,\tau)$ (i.e., the topology of $\mathcal{F}$ happens to be a subspace topology of a topology on $X^Y$). \label{SupSsFtNt}(\textbf{footnote}\footnote{
If $A$ is a set and $(B,\tau)$ a space such that $A\subset B$, then the \textbf{subspace topology} on $A$ is the topology $\tau^i$ that makes the inclusion $i:(A,\tau^i)\hookrightarrow (B,\tau)$ a quotient map, where we know $\tau^i=\tau\cap A:=\{U\cap A:U\in\tau\}$. Similarly, if $(A,\tau)$ is a space and $B$ a set such that $A\subset B$, then the inclusion $i:A\hookrightarrow B$ induces the $i$-quotient topology (call it the \textbf{superspace topology}) $\tau_i$ on $B$ given by $\tau_i=\{U\subset B:U\cap A\in\tau\}$.
})
\end{rmk}

\section{\textnormal{\bf Subset hyperspaces as quotients of function spaces}}\label{SHQFS}
\noindent We will characterize a class of subset hyperspaces as explicit quotients of function spaces. To have reasonable separation properties, we aim to choose the function space topologies to be $\tau_p$-compatible and likewise choose the hyperspaces topologies to be $\tau_v$-compatible.

\begin{dfn}[{Indexed subset hyperspaces: Limit Vietoris topology, Unordering map}]\label{ProdDescDfn}
Let $Y$ be a set. Given a family of spaces $\mathcal{X}=\{X_y:y\in Y\}$, let $X=\bigcup\mathcal{X}:=\bigcup_{y\in Y}X_y$, and let
\[
\textstyle \prod\mathcal{X}=\prod_{y\in Y}X_y:=\{\textrm{maps}~f:Y\rightarrow X,~y\mapsto f_y\in X_y\}=\{(f_y)_{y\in Y}:f_y\in X_y\}
\]
be their Cartesian product as sets. Recall that the \textbf{product topology} $\tau_p$ on ${\prod\mathcal{X}}$ has base sets
\[
\textstyle[O_F]_p=[\{O_y:y\in F\}]_p:=\big\{f\in {\prod\mathcal{X}}:f_y\in O_y~\forall y\in F\big\}=\prod_{y\in F}O_y~\times~\prod_{y\in Y\backslash F}X_{y},
\]
for finite subsets $F\subset Y$ and open subset collections $O_F=\{O_y\subset X_y: y\in F\}$.

Let us give $X$ the topology $\mathcal{O}(X):=\{T\subset X:T\cap X_y\subset X_y~\textrm{is open}~\forall y\in Y\}$, and call it the \textbf{limit topology} on $X$. Also, let us give $\mathcal{P}^\ast(X):=\mathcal{P}(X)\backslash\{\emptyset\}$ the topology $\tau_v$ (and call it the \textbf{limit Vietoris topology} on $\mathcal{P}^\ast(X)$) with a base of sets of the form
\begin{align*}
[T_F]_v=[\{T_{\alpha}:{\alpha}\in F\}]_v&:=\textstyle\{A\in\mathcal{P}^\ast(X):A\subset\bigcup_{\alpha} T_{\alpha},~A\cap T_{\alpha}\neq\emptyset~\forall {\alpha}\in F\},
\end{align*}
for finite collections $~T_F=\{T_{\alpha}:{\alpha}\in F\}\subset\mathcal{O}(X)$, where $F$ is an arbitrary finite set. (\textbf{footnote}\footnote{The collection $\mathcal{B}:=\{[T_F]_v:F~\textrm{finite}\}$ indeed forms a base for a topology on $\mathcal{P}^\ast(X)$, because if we let $T_F=\{T_1,...,T_n\}$ and $T'_{F'}=\{T'_1,...,T'_{n'}\}$, then $[T_F]_v\cap [T'_{F'}]_v=\big[\{(\bigcup_iT_i)\cap T_1',\cdots,(\bigcup_iT_i)\cap T_{n'}',T_1\cap(\bigcup_jT'_j),\cdots,T_n\cap(\bigcup_jT'_j)\}\big]_v
=\big[\big\{(\bigcup T_F)\cap T'_{\alpha'}:\alpha'\in F'\big\}\cup\big\{T_\alpha\cap\bigcup T'_{F'}:\alpha\in F\big\}\big]_v$, which also lies in $\mathcal{B}$.
}).

Let us now define the following items (where items (\ref{HPSpDfnIt1})-(\ref{HPSpDfnIt4}) have already appeared in the introduction under a less general setup):
\begin{enumerate}[leftmargin=0.8cm]
\item\label{HPSpDfnIt1} The set of \textbf{nonempty closed subsets} of $X$: $Cl(X)\subset \mathcal{P}^\ast(X)$.
\item\label{HPSpDfnIt2} The set of \textbf{nonempty compact subsets} of $X$: $K(X)\subset Cl(X)$.
\item\label{HPSpDfnIt3} The set of \textbf{nonempty bounded closed subsets} of $X$ (if $X$ is a metric space): $BCl(X)\subset Cl(X)$.
\item\label{HPSpDfnIt4} The \textbf{unordering map} $q:{\prod\mathcal{X}}\rightarrow Cl(X),~f\mapsto cl_X(f(Y))$.

Given a $j$-labeled topology $\tau_j$ on $\mathcal{F}\subset {\prod\mathcal{X}}$, we denote by $\tau_{jq}$ the topology (called the \textbf{$\tau_j$-quotient topology}) on $q(\mathcal{F})\subset Cl(X)$ induced by the restriction $q:(\mathcal{F},\tau_j)\rightarrow q(\mathcal{F})=(q(\mathcal{F}),\tau_{jq})$ as a quotient map. (\textbf{footnote}\footnote{
Recall that $\tau_{jq}$ is the strongest topology $\tau_{jc}$ on $q(\mathcal{F})\subset {\prod\mathcal{X}}$ such that $q:(\mathcal{F},\tau_j)\rightarrow (q(\mathcal{F}),\tau_{jc})$ is continuous. In particular, if $\tau_j:=\tau\cap\mathcal{F}$ (i.e., $\tau_j$ is the $\tau$-subspace topology of $\mathcal{F}$ associated with some topology $\tau$ on ${\prod\mathcal{X}}$) then $\tau_q\cap q(\mathcal{F})\subset\tau_{jq}=(\tau\cap\mathcal{F})_q$ (where $\tau_q$ is the $\tau$-quotient topology of $q({\prod\mathcal{X}})$ and $\tau_q\cap q(\mathcal{F})$ is the $\tau_q$-subspace topology of $q(\mathcal{F})\subset q({\prod\mathcal{X}})$). Therefore, continuity of $q:(\mathcal{F},\tau\cap\mathcal{F})\rightarrow (q(\mathcal{F}),(\tau\cap \mathcal{F})_q)$ implies continuity of $q:(\mathcal{F},\tau\cap\mathcal{F})\rightarrow (q(\mathcal{F}),\tau_q\cap q(\mathcal{F}))$.
})
\item The set of \textbf{$Y$-indexed closed subsets} of $X$: $Cl_Y(\mathcal{X}):=Cl(X)\cap q({\prod\mathcal{X}})\stackrel{(\ast)}{=}q({\prod\mathcal{X}})$ (where step ($\ast$) is due to the use of closure in the definition of $q$). Also, let
    \[
    \textstyle Cl(Y,\mathcal{X}):=q^{-1}(Cl_Y(\mathcal{X}))\stackrel{(\ast)}{=}{\prod\mathcal{X}}.
    \]
\item The set of \textbf{$Y$-indexed compact subsets} of $X$: $K_Y(\mathcal{X}):=K(X)\cap q({\prod\mathcal{X}})$. Also, let
    \[
    \textstyle K(Y,\mathcal{X}):=q^{-1}(K_Y(\mathcal{X}))~~\stackrel{\textrm{$q$-saturated}}{\subset}~~{\prod\mathcal{X}}.
    \]
\item The set of \textbf{$Y$-indexed bounded closed subsets} of $X$ (if $X$ is a metric space): $BCl_Y(\mathcal{X}):=BCl(X)\cap q({\prod\mathcal{X}})$. Also, let
    \[
    \textstyle BCl(Y,\mathcal{X}):=q^{-1}(BCl_Y(\mathcal{X}))~~\stackrel{\textrm{$q$-saturated}}{\subset}~~{\prod\mathcal{X}}.
    \]
\item The set of \textbf{nonempty finite subsets} of $X$: $FS(X):=\{A\in K(X):|A|<\infty\}$.
\item The set of \textbf{nonempty $n$-finite subsets} of $X$: $FS_n(X):=\{A\in FS(X):|A|\leq n\}$.
\end{enumerate}
When $\mathcal{X}=\{X\}$, i.e., $X_y=X$ for all $y\in Y$, we will replace $\mathcal{X}$ with $X$ in the $Y$-indexed subset/function spaces above by setting
\begin{align*}
&Cl_Y(X):=Cl_Y(\mathcal{X}),~~Cl(Y,X):=Cl(Y,\mathcal{X}),~~K_Y(X):=K_Y(\mathcal{X}),~~ K(Y,X):=K(Y,\mathcal{X}),\\
&BCl_Y(X):=BCl_Y(\mathcal{X}),~~BCl(Y,X):=BCl(Y,\mathcal{X}).
\end{align*}
\end{dfn}

\begin{note}
If $X$ is $T_1$ (i.e., singletons of $X$ are closed in $X$) then $X\cong F_1(X)=\textrm{Singlt}(X)$.
\end{note}

\begin{dfn}[{$q$-full subset}]
In the setup of Definition \ref{ProdDescDfn}, a subset {\footnotesize$\mathcal{F}\subset\prod\mathcal{X}$} is \textbf{$q$-full} if
\[\textstyle
q(\mathcal{F})=q(\prod\mathcal{X}).
\]
\end{dfn}
Notice that if $X,Y$ are sets and $\mathcal{F}\subset X^Y$ is $q$-full, then we have an injection $X\hookrightarrow\mathcal{F},~x\mapsto c_x$, where $c_x$ is the \textbf{constant map} $Y\rightarrow X,~y\mapsto x$. Consequently, if $X$ is a space, then topologies on $\mathcal{F}$ (just like hypertopologies of a $T_1$ space $X$) may be seen as extensions/generalizations of the topology of $X$.

\begin{dfn}[{rc-topology, rc-space}]\label{SrcTopDfn}
In the setup of Definition \ref{ProdDescDfn}, fix a $q$-full subset $\mathcal{F}=\mathcal{F}(Y,\mathcal{X})\subset\prod\mathcal{X}$ and consider a topology $\tau$ on $\mathcal{F}$. A net $f_\alpha\in (\mathcal{F},\tau)$ is \textbf{compactly-ranged} (hence a \textbf{cr-net}) if there exists a compact set {\footnotesize $K\subset X=\bigcup_{y\in Y} X_y$} and a tail {\footnotesize $T_\beta:=\{f_\alpha:\alpha\geq\beta\}$} of $f_\alpha$ such that $\bigcup q(T_\beta):=\bigcup_{\alpha\geq\beta} q(f_\alpha)\subset K$.

The topology $\tau$ on $\mathcal{F}$ is \textbf{range-compact} (hence an \textbf{rc-topology} on $\mathcal{F}$, making $(\mathcal{F},\tau)$ an \textbf{rc-space}) if every cr-net $f_\alpha\in (\mathcal{F},\tau)$ has a convergent subnet.
\end{dfn}

\begin{rmk}\label{SrcTopExist}
Let $X$ be a space and $Y$ a set. Then, by Tychonoff's product theorem, $(X^Y,\tau_p)$ is an rc-space: Indeed, if $K\subset X$ is compact, then $K^Y\subset(X^Y,\tau_p)$ is a compact subspace. We also recall that in a compact space, a sequence (being a net) has a convergent subnet, but not necessarily a convergent subsequence, i.e., a compact space need not be sequentially compact.
\end{rmk}

\begin{dfn}[{wrc-topology, wrc-space}]\label{WsrcTopDfn}
In the setup of Definition \ref{ProdDescDfn}, let $\tau'$ be a topology on $Cl_Y(\mathcal{X})$, and fix a $q$-full subset $\mathcal{F}=\mathcal{F}(Y,\mathcal{X})\subset\prod\mathcal{X}$. A topology $\tau$ on $\mathcal{F}$ is $\tau'$-\textbf{weakly range-compact} (hence a $\tau'$-\textbf{wrc-topology} on $\mathcal{F}$, making $(\mathcal{F},\tau)$ a $\tau'$-\textbf{wrc-space}) if for any $\tau'$-convergent net $C_\alpha\in (K_Y(\mathcal{X}),\tau')$ such that $\bigcup_{\alpha\geq\beta} C_\alpha\subset K$ for a compact set $K\subset X$ and some index $\beta$, every net $g_\alpha\in q^{-1}(C_\alpha)$, $\alpha\geq\beta$, has a $\tau$-convergent subnet in $(\mathcal{F},\tau)$. (\textbf{footnote}\footnote{
Observe that the rc-property is stronger than (i.e., implies) the wrc-property. That is, every rc-topology (such as $\tau_p$) is a wrc-topology.
}).

A $\tau_v$-wrc-topology (resp., $\tau_v$-wrc-space) will simply be called a \textbf{wrc-topology} (resp., \textbf{wrc-space}).
\end{dfn}

\begin{dfn}[{swrc-topology, swrc-space, Standard topology of the indexed subset hyperspaces}]\label{StdTopIndSSP}
In the setup of Definition \ref{ProdDescDfn}, fix a $q$-full subset $\mathcal{F}=\mathcal{F}(Y,\mathcal{X})\subset\prod\mathcal{X}$. A wrc-topology $\tau_{\pi}$ on $\mathcal{F}$ is a \textbf{standard wrc-topology} (hence a \textbf{swrc-topology} on $\mathcal{F}$, making $(\mathcal{F},\tau_{\pi})$ a \textbf{swrc-space}) if $\tau_v\subset\tau_{\pi q}$ in $K_Y(\mathcal{X})$ (where $\tau_{\pi q}$ denotes the \textbf{$\tau_{\pi}$-quotient topology} on $Cl_Y(\mathcal{X})$), i.e., if with
\[
\textstyle K\mathcal{F}(Y,\mathcal{X}):=\mathcal{F}\cap K(Y,\mathcal{X})~\subset~\bigcup_{K\in K(X)}K^Y,
\]
the map
\[
\textstyle q|_{K\mathcal{F}(Y,\mathcal{X})}:K\mathcal{F}(Y,\mathcal{X})\subset(\mathcal{F},\tau_\pi)\rightarrow\big(Cl_Y(\mathcal{X}),\tau_v)
\]
is continuous.

We will give $Cl_Y(\mathcal{X})$ the $\tau_{\pi}$-quotient topology $\tau_{\pi q}$ (as our \textbf{standard topology} on $Cl_Y(\mathcal{X})$).
\end{dfn}

\begin{question}\label{StdTopQsn}
Let $X$ be a space, $Y$ a set (and let $\mathcal{X}:=\{X\}$), and $\mathcal{F}\subset X^Y$ a $q$-full subset. How do we explicitly specify the preferred topology $\tau_\pi$ of Definition \ref{StdTopIndSSP}? Theorem \ref{ScpTopThm} specifies $\tau_\pi$ for some special cases where $\mathcal{F}$ is sufficiently well-behaved.
\end{question}

\begin{dfn}[{Set-open topology, Compact-open topology}]\label{SetOpTop}
Let $X$ be a space, $Y$ a set, and $\mathcal{S}\subset\mathcal{P}(Y)$ a family of subsets of $Y$. The \textbf{$\mathcal{S}$-open topology} $\tau_{\mathcal{S}}$ on $X^Y$ is the topology with a subbase given by the sets
\[
[S,O]_{\mathcal{S}}:=\{f\in X^Y:f(S)\subset O\},~~~~\textrm{for sets $S\in\mathcal{S}$ and open sets $O\subset X$}.
\]
In particular, if $Y$ is a space and $\mathcal{S}=K(Y)$, then $\tau_\mathcal{S}$ is called the \textbf{compact-open topology}, $\tau_{co}$, on $X^Y$.
\end{dfn}

\begin{rmk}\label{TpiChRmk}
\begin{enumerate}[leftmargin=0.8cm]
\item \label{TpiChRmk2} For any $T_3$-space (i.e., regular Hausdorff space) $X$, according to \cite[Theorem 2.4]{bartsch2014}, there always exists a compact space $Y$ such that $q\big(C(Y,X)\big)=K(X)$ (where $C(Y,X):=\{\textrm{continuous}~f\in X^Y\}$) and we have a quotient map
    \[
    q:C(Y,X)\subset (X^Y,\tau_{\textrm{co}})\rightarrow (K(X),\tau_v).
    \]
\item Based on the setup in Definition \ref{StdTopIndSSP}, if $X,Y$ are spaces and $\tau_\pi$ is a swrc-topology on a $q$-full $\mathcal{F}\subset X^Y$, then we have the quotient map
    \[
    q:(\mathcal{F},\tau_{\pi})\rightarrow (Cl_Y(X),\tau_{\pi q}).
    \]
   In this special case (where both $X$ and $Y$ are spaces), candidates for $\tau_{\pi}$ in Definition \ref{StdTopIndSSP} are $q^{-1}(\tau_v)$ and $\tau_{co}$, and it might therefore be reasonable/sufficient (regarding Question \ref{StdTopQsn}) to search for a $\tau_p$-compatible $\tau_{\pi}$ in the range (\textbf{footnote}\footnote{
Consider the map $q:(X^Y,\tau_1,\tau_2)\rightarrow (Cl_Y(X),\tau_{1q},\tau_{2q})$. If $\tau_1\subset\tau_2$, then $B\in \tau_{1q}$ $\iff$ $q^{-1}(B)\in\tau_1\subset\tau_2$, $\Rightarrow$ $q^{-1}(B)\in\tau_2$, $\iff$ $B\in \tau_{2q}$ (i.e., $\tau_1\subset \tau_2$ $\Rightarrow$ $\tau_{1q}\subset\tau_{2q}$).
})
\[
q^{-1}(\tau_v)\subset\langle\tau_p\cup \tau_{s}\rangle\subset\tau_{\pi}\subset\tau_{co}
\]
(in accordance with ~ $\tau_v\subset\langle\tau_p\cup\tau_{s}\rangle_q\subset\tau_{\pi q}\subset\tau_{co\!\!\!~q}$ ~ due to Theorem \ref{QuotRepVTop} and \cite[Corollary 2.3]{bartsch2014}), where $\tau_{s}$ is the topology of symmetric convergence (Definition \ref{TopGlobConv}).

In addition to Theorem \ref{ScpTopThm}, ways of choosing $\tau_{\pi}$ might be found in \cite[Lemma 2.2]{bartsch2014}, \cite[Corollary 3.18]{bartsch2004}, and \cite[]{bartsch2002}, and potentially involve generalizations of Tychonoff's product theorem (e.g., in \cite{park1975,FoxMor1974,kimber1974}).
\end{enumerate}
\end{rmk}

We will now discuss the Vietoris topology (a more general version of which we have already introduced in Definition \ref{ProdDescDfn}) in sufficient detail for our subsequent discussion. Using convergence of nets, a more intuitive interpretation of the Vietoris topology is given later in Definition \ref{VietTopInt}.

\begin{dfn}[{Vietoris topology}]\label{VietorisTop}
Let $X$ be a space, $Y$ a set, and $\mathcal{F}\subset X^Y$ a $q$-full subset. For any open set $O\subset X$ and any finite collection of open subsets $O_F=\{O_{\alpha}:{\alpha}\in F\}$ of $X$,
let
\begin{align*}
[O_F]_v&\textstyle:=\{C\in Cl_Y(X):C\subset\bigcup_{{\alpha}\in F}O_{\alpha},~C\cap O_{\alpha}\neq\emptyset,\forall {\alpha}\in F\}\nonumber\\
&\textstyle=\left\{q(f):f\in \mathcal{F},~q(f)\subset\bigcup_{{\alpha}\in F}O_{\alpha},~q(f)\cap O_{\alpha}\neq\emptyset,\forall {\alpha}\in F\right\},\\
O^+&\textstyle:=[\{O\}]_v=\{C\in Cl_Y(X):C\subset O\}=\left\{q(f):f\in \mathcal{F},~q(f)\subset O\right\},\\
O^-&\textstyle:=[\{X,O\}]_v=[\{O^c\}]_v{}^c=\{C\in Cl_Y(X):C\cap O\neq\emptyset\}=\left\{q(f):f\in \mathcal{F},~q(f)\cap O\neq\emptyset\right\}\\
&\textstyle=\left\{q(f):f\in \mathcal{F},~f(Y)\cap O\neq\emptyset\right\}.
\end{align*}
The \textbf{Vietoris topology} $\tau_v$ (also see Definition \ref{ProdDescDfn}) of $Cl_Y(X)$ is the topology with base
\begin{align*}
&\mathcal{B}_v=\big\{[O_F]_v:\textrm{$F$ a finite set, $O_{\alpha}\subset X$ open, $\forall {\alpha}\in F$}\big\}.
\end{align*}
Let the \textbf{upper Vietoris topology} $\tau_v^+$ (resp., \textbf{lower Vietoris topology} $\tau_v^-$) be generated by
\begin{align*}
&S\mathcal{B}_v^+=\big\{O^+:\textrm{$O\subset X$ open}\big\}~~\big(\textrm{resp.},~~S\mathcal{B}_v^-=\big\{O^-:\textrm{$O\subset X$ open}\big\}\big).
\end{align*}
The Vietoris topology then satisfies $\tau_v=\langle S\mathcal{B}_v^-\cup S\mathcal{B}_v^+\rangle=\langle\tau_v^-\cup\tau_v^+\rangle$.
\end{dfn}

\begin{dfn}[{Topology of symmetric convergence}]\label{TopGlobConv}
Let $X$ be a space and $Y$ a set. The \textbf{topology of symmetric convergence}, $\tau_{s}$, on $X^Y$ is the topology with base sets
\[
[O]_s:=\{f\in X^Y:q(f)\subset O\},~~~~\textrm{for open sets}~O\subset X.
\]
\end{dfn}

\begin{dfn}[{Finitely $q$-stable set of functions}]
Let $X$ be a space and $Y$ a set. A set of functions $\mathcal{F}\subset X^Y$ is \textbf{finitely $q$-stable} if for any $f\in\mathcal{F}$, any finite set $F\subset Y$, and any injection $\sigma:F\rightarrow Y$, there exists $g=g_{F,\sigma}\in\mathcal{F}$ such that $f|_F=g|_{\sigma(F)}$ and $q(f)=q(g)$.
\end{dfn}

\begin{lmm}[{Lower quotient topology}]\label{QuotVietCont}
Let $X$ be a space, $Y$ a set, and $\mathcal{F}\subset X^Y$ a q-full finitely $q$-stable subset. Then with respect to $\mathcal{F}$, we have
\begin{enumerate}[leftmargin=0.9cm]
\item[(i)~] $\tau_{pq}=\tau_v^-$ in $Cl_Y(X)$ (i.e., the map $q:(\mathcal{F},\tau_p)\rightarrow(Cl_Y(X),\tau_v^-)$ is a quotient map) and
\item[(ii)] $q^{-1}(\tau_v^-)\subsetneq \tau_p$ unless $Y$ is a singleton.
\end{enumerate}
\end{lmm}
\begin{proof}
(i) \noindent\textbf{Proving $\tau_{pq}\subset\tau_v^-$}:
A set $\mathcal{A}\subset Cl_Y(X)$ is $\tau_{pq}$-open iff $q^{-1}(\mathcal{A})=\left\{f\in \mathcal{F}:q(f)\in \mathcal{A}\right\}$ is $\tau_p$-open, i.e., iff there exist a collection of finite sets $\{F_i\subset Y\}_{i\in I}$ and open sets $\{O_y^i\subset X:i\in I,y\in F_i\}=\bigcup_{i\in I}\{O_y^i\subset X:y\in F_i\}=\bigcup_{i\in I}O^i_{F_i}$ such that
\begin{align*}
\textstyle q^{-1}(\mathcal{A})&\textstyle=\bigcup_{i\in I}[O_{F_i}^i]_p=\bigcup_{i\in I}\{f\in \mathcal{F}:~f(y)\in O_y^i,\forall y\in F_i\},
\end{align*}
which is the general form of a $\tau_p$-open set in $\mathcal{F}$. By applying $q$ (and noting $q:\mathcal{F}\rightarrow Cl_Y(X)$ is surjective, hence $q(q^{-1}(\mathcal{A}))=\mathcal{A}$), we get
\begin{align*}
\mathcal{A}&\textstyle=q(q^{-1}(\mathcal{A}))=\bigcup_{i\in I}\left\{q(f):f\in \mathcal{F},~f(y)\in O^i_y,\forall y\in F_i\right\}\nonumber\\
&\textstyle\stackrel{(\ast)}{=}\bigcup_{i\in I}\left\{q(f):f\in \mathcal{F},~f(Y)\cap O_y^i\neq\emptyset,\forall y\in F_i\right\}\nonumber\\
&\textstyle=\bigcup_{i\in I}\left\{q(f):f\in \mathcal{F},~q(f)\cap O_y^i\neq\emptyset,\forall y\in F_i\right\}=\bigcup_{i\in I}\bigcap_{y\in F_i}(O_y^i)^-\nonumber\\
&\textstyle\in\tau_v^-,
\end{align*}
where at step ($\ast$), $\subset$ is obvious and $\supset$ follows from finite $q$-stability of $\mathcal{F}$.

\noindent\textbf{Proving $\tau_{pq}\supset\tau_v^-$ (i.e., $q:(\mathcal{F},\tau_p)\rightarrow (Cl_Y(X),\tau_v^-)$ is continuous)}: Let $O\subset X$ be open. Then
\begin{align*}
q^{-1}(O^-)&\textstyle=q^{-1}\left\{q(f):f\in \mathcal{F},q(f)\cap O\neq\emptyset\right\}
=\bigcup_{f\in\mathcal{F}}\left\{q^{-1}(q(f)):q(f)\cap O\neq\emptyset\right\}\\
&\textstyle=\bigcup_{f\in\mathcal{F}}\left\{g\in \mathcal{F}:q(g)=q(f),q(f)\cap O\neq\emptyset\right\}
=\left\{g\in \mathcal{F}:q(g)\cap O\neq\emptyset\right\}\\
&\textstyle=\left\{g\in \mathcal{F}:g(Y)\cap O\neq\emptyset\right\}=\left\{g\in \mathcal{F}:\exists y\in Y,g(y)\in O\right\}=\bigcup_{y\in Y}\left\{g\in \mathcal{F}:g(y)\in O\right\}\\
&\textstyle=\bigcup_{y\in Y}[(y,O)]_p\in\tau_p.
\end{align*}

(ii) This follows from the observation that the subbase elements $\{[y,O]_p:y\in Y\}$ of $\tau_p$ can distinguish the points of $Y$, meanwhile, by construction, neither the elements of $q^{-1}(\tau_v^-)$ nor those of $q^{-1}(\tau_v^+)$ can distinguish the points of $Y$. Indeed, if $|Y|\geq 2$, pick $y_1,y_2\in Y$ such that $y_1\neq y_2$ and $[y_1,O]_p\neq [y_2,O]_p$. Then each $[y_i,O]_p$ (an element of $\tau_p$) depends asymmetrically on $y_1$ and $y_2$. But every member of $q^{-1}(\tau^-_v)$ depends symmetrically on $y_1$ and $y_2$ (as the expression for $q^{-1}(O^-)$ above shows), and so $q^{-1}(\tau^-_v)$ does not contain $[y_i,O]_p$. That is, if $|Y|\geq2$, then $\tau_p\backslash q^{-1}(\tau_v^-)\neq\emptyset$ (provided $X$ is a sufficiently nontrivial space).
\end{proof}

\begin{rmk}\label{QuotVietContRmk}
The following are further observations following the proof of Lemma \ref{QuotVietCont}.
\begin{enumerate}[leftmargin=0.7cm]
\item For $y\in Y$ and any open set $O\subset X$, we have both $q([O]_s)=O^+$ and
\[
q([(y,O)]_p)=\left\{q(f):f\in \mathcal{F},f(y)\in O\right\}\stackrel{\textrm{finite $q$-stability}}{=}\left\{q(f):f\in\mathcal{F},q(f))\cap O\neq\emptyset\right\}=O^-.
\]
However, these relations alone do not guarantee openness of $q:(\mathcal{F},\tau)\rightarrow (Cl_Y(X),\tau_v)$ whether for $\tau=\tau_p$, $\tau=\tau_{s}$, or $\tau=\langle\tau_p\cup\tau_{s}\rangle$, since $q$ (like other maps in general) need not preserve finite intersections.
\item\label{QVCRmkItem3} $q^{-1}(\tau_v^+)=\tau_{s}$ and, if $X$ is $T_1$, then $\tau_{s}\subset\tau_p$ $\iff$ $Y$ is finite, since
{\small\begin{align*}
q^{-1}(O^+)&\textstyle=q^{-1}\left\{q(f):f\in \mathcal{F},q(f)\subset O\right\}=\bigcup\limits_{f\in\mathcal{F}}\left\{q^{-1}(q(f)):q(f)\subset O\right\}=\bigcup\limits_{f\in\mathcal{F}}\left\{g\in \mathcal{F}:q(g)=q(f)\subset O\right\}\\
&\textstyle=\left\{g\in \mathcal{F}:q(g)\subset O\right\}=[O]_s=\left\{g\in \mathcal{F}:g(Y)\subset O\right\}\cap\left\{g\in \mathcal{F}:\partial g(Y)\subset O\right\}\\
&\textstyle=\bigcap_{y\in Y}\left\{g\in \mathcal{F}:g(y)\in O\right\}\cap\left\{g\in \mathcal{F}:\partial g(Y)\subset O\right\}=\left[\bigcap_{y\in Y}[(y,O)]_p\right]\cap\left\{g\in \mathcal{F}:\partial g(Y)\subset O\right\}\\
&\textstyle\in\tau_p~~\iff~~\textrm{$Y$ is finite}.
\end{align*}}
\item So, if $X$ is $T_1$, then on $\mathcal{F}\subset X^Y$ ($X$ nontrivial), $~q^{-1}(\tau_v)\subset\tau_p$ $\iff$ $Y$ is finite.
\end{enumerate}
\end{rmk}

\begin{lmm}[{Upper quotient topology}]\label{QuotVietContU}
Let $X$ be a space, $Y$ a set, and $\mathcal{F}\subset X^Y$ a $q$-full subset. We have:
\begin{enumerate}[leftmargin=0.9cm]
\item[(i)~~] $\tau_{sq}=\tau_v^+$ in $Cl_Y(X)$ (i.e., the map $q:(\mathcal{F},\tau_{s})\rightarrow(Cl_Y(X),\tau_v^+)$ is a quotient map),
\item[(ii)~] $\tau_{s}=q^{-1}(\tau_v^+)$, and
\item[(iii)] If $X$ is $T_1$, then $\tau_{s}\subset\tau_p$ $\iff$ $Y$ is finite (by Lemma \ref{QuotVietCont}'s Remark (\ref{QVCRmkItem3})).
\end{enumerate}
\end{lmm}
\begin{proof}
(i) \noindent\textbf{Proving $\tau_{sq}\subset\tau_v^+$}:
A set $\mathcal{A}\subset Cl_Y(X)$ is $\tau_{sq}$-open iff $q^{-1}(\mathcal{A})=\left\{f\in \mathcal{F}:q(f)\in \mathcal{A}\right\}$ is $\tau_{s}$-open, i.e., iff there exist a collection of open sets $\{O_i\subset X:i\in I\}$ such that
\begin{align*}
\textstyle q^{-1}(\mathcal{A})&\textstyle=\bigcup_{i\in I}[O_i]_s=\bigcup_{i\in I}\{f\in\mathcal{F}:~q(f)\in O_i\},
\end{align*}
which is the general form of a $\tau_{s}$-open set in $\mathcal{F}\subset X^Y$. By applying $q$ (and noting $q:\mathcal{F}\rightarrow Cl_Y(X)$ is surjective, hence $q(q^{-1}(\mathcal{A}))=\mathcal{A}$) we get
\begin{align*}
\mathcal{A}&\textstyle=q(q^{-1}(\mathcal{A}))=\bigcup_{i\in I}\left\{q(f):f\in \mathcal{F},~q(f)\subset O_i\right\}=\bigcup_{i\in I}O_i^+\in\tau_v^+.
\end{align*}

\noindent\textbf{Proving  $\tau_{sq}\supset\tau_v^+$ (i.e., {\footnotesize $q:(\mathcal{F},\tau_{s})\rightarrow (Cl_Y(X),\tau_v^+)$} is continuous)}: As in Lemma \ref{QuotVietCont}'s Remark (\ref{QVCRmkItem3}),
\begin{align*}
q^{-1}(O^+)&\textstyle=q^{-1}\left\{q(f):f\in \mathcal{F},q(f)\subset O\right\}=\bigcup_{f\in\mathcal{F}}\left\{q^{-1}(q(f)):q(f)\subset O\right\}\\
&\textstyle=\bigcup_{f\in\mathcal{F}}\left\{g\in \mathcal{F}:q(g)=q(f)\subset O\right\}\\
&\textstyle=\left\{g\in \mathcal{F}:q(g)\subset O\right\}=[O]_s\in\tau_{s}.
\end{align*}

(ii) From the above equality, we see that $\tau_{s}=q^{-1}(\tau_v^+)$.
\end{proof}

For a finite $Y$, the following theorem (Theorem \ref{QuotRepVTop}) realizes $\tau_v$ as a quotient of a $\tau_p$-compatible topology. For a general $Y$, the realization of $\tau_v$ as a of a $\tau_p$-compatible topology will be accomplished in Theorem \ref{ScpTopThmPrv} (which requires Theorem \ref{QuotRepVTop}).

\begin{thm}[{$\tau_v$-compatibility of $\langle\tau_p\cup\tau_{s}\rangle_q$}]\label{QuotRepVTop}
Let $X$ be a space, $Y$ a set, $\mathcal{F}\subset X^Y$ a $q$-full finitely $q$-stable subset, and consider the map $q:\mathcal{F}\rightarrow Cl_Y(X)$. The following are true:
\begin{enumerate}[leftmargin=0.9cm]
\item\label{VietQRep1} $\tau_v=\langle\tau_{pq}\cup\tau_{sq}\rangle\subset\langle\tau_p\cup\tau_{s}\rangle_q$,~ where if $X$ is $T_1$, then equality holds iff $Y$ is finite (in which case $\tau_v=\tau_{pq}$).
\item\label{VietQRep2} $q^{-1}(\tau_v)=\langle q^{-1}(\tau_v^-)\cup q^{-1}(\tau_v^+)\rangle\subset\langle\tau_p\cup\tau_{s}\rangle$,~ where equality holds iff $Y$ is a singleton.
    \end{enumerate}
In particular, the map ~$q:(\mathcal{F},\langle\tau_p\cup\tau_{s}\rangle)\rightarrow (Cl_Y(X),\tau_v)$~ is continuous.
\end{thm}
\begin{proof}
By Lemmas \ref{QuotVietCont} and \ref{QuotVietContU}, we get both (\ref{VietQRep1}) via $\tau_v=\langle \tau_v^-\cup\tau_v^+\rangle$ and (\ref{VietQRep2}) via $q^{-1}(\tau_v)=q^{-1}(\langle\tau_v^-\cup\tau_v^+\rangle)=\langle q^{-1}(\tau_v^-)\cup q^{-1}(\tau_v^+)\rangle$.
\end{proof}

\begin{rmk}\label{SSpConnect}
Let $X$ be a space, $Y$ a set, and $(\mathcal{F},\tau)\subset X^Y$ a $q$-full function space.
\begin{enumerate}[leftmargin=0.7cm]
\item\label{SSpConnect0} In the above results for $q:(\mathcal{F},\tau)\rightarrow (Cl_Y(X),\tau_q)$ with $\tau_q=\tau_v$, if $FS_n(X)\subset (Cl_Y(X),\tau_q)$ is closed, then (by Definition \ref{QuotSpDfn}'s Remark (\ref{QuotTopCautIt3})) $q:(q^{-1}(FS_n(X))\cap\mathcal{F},\tau)\subset(\mathcal{F},\tau)\rightarrow (FS_n(X),\tau_q)$ is a quotient map as well. In this case, Theorem \ref{QuotRepVTop}(\ref{VietQRep1}) holds (i) with $\mathcal{F}$ and $Cl_Y(X)$ replaced by $q^{-1}(FS_n(X))\cap\mathcal{F}$ and $FS_n(X)$ respectively, and (ii) with equality for any $Y$.
\item If $(\mathcal{F},\tau)$ is compact, connected, or path-connected, then so is $(Cl_Y(X),\tau_q)$, since $q:(\mathcal{F},\tau)\rightarrow (Cl_Y(X),\tau_q)$ is continuous.
\item Continuous maps $h:(Cl_Y(X),\tau_q)\rightarrow Z$ (for a space $Z$) are precisely continuous maps $h:(\mathcal{F},\tau)\rightarrow Z$ that are constant on the equivalence classes $[f]:=q^{-1}(q(f))=\left\{g\in \mathcal{F}:q(g)=q(f)\right\}$ (i.e., $h|_{[f]}=const$) for all $f\in \mathcal{F}$.
\item Given topologies $\tau_1\subset\tau_2$ (e.g., topologies of metrics $d_1\leq d_2$) on $X$, convergence of a net $x_\alpha\in (X,\tau_2)$ (resp., compactness of a set $A\subset(X,\tau_2)$) implies convergence of $x_\alpha\in (X,\tau_1)$ (resp., compactness of $A\subset(X,\tau_1)$). That is, the two sets of convergent nets satisfy
    \[
    \tau_1\subset\tau_2~~\iff~~\textrm{ConvNet}(\tau_2)\subset\textrm{ConvNet}(\tau_1).
    \]
    Similarly, continuity of a map $f\in(X,\tau_2)^{(Y,\tau_Y)}$ implies continuity of $f\in(X,\tau_1)^{(Y,\tau_Y)}$, i.e.,
    \[
    \tau_1\subset\tau_2~~\iff~~C\big(Y,(X,\tau_2)\big)\subset C\big(Y,(X,\tau_1)\big).
    \]
\item\label{SSpConnect4} Let $G$ be a \textbf{topological group} (i.e., a group that is a topological space with continuous group-multiplication and inversion) and $H\subset G$ a closed normal subgroup. For each $g\in G$, let $L_g:G\rightarrow G,~x\mapsto gx$ be left translation by $g$, and $Lt(H,G):=\{L_g|_H:g\in G\}\subset G^H$. In the anticipated quotient map ~$q:Lt(H,G)\rightarrow q(Lt(H,G))\subset Cl_H(G),~f\mapsto\overline{f(H)}$,~ we have
    \[
    q(Lt(H,G))=\left\{\overline{L_g(H)}=gH:g\in G\right\}=G/H.
    \]
\item\label{SSpConnect5} Let $R$ be a \textbf{topological ring} (i.e., a ring that is a topological space with continuous multiplication and addition) and $I\subset R$ a closed ideal. For each $r\in R$, let $L_r:R\rightarrow R,~x\mapsto r+x$ be translation by $r$, and $Lt(I,R):=\{L_r|_I:r\in R\}\subset R^I$. In the anticipated quotient map ~$q:Lt(I,R)\rightarrow q(Lt(I,R))\subset Cl_I(R),~f\mapsto\overline{f(I)}$,~ we have
    \[
    q(Lt(I,R))=\left\{\overline{L_r(I)}=r+I:r\in R\right\}=R/I.
    \]
\item\label{SSpConnect6} Let $R$ be a topological ring and $_RM$ a \textbf{topological $R$-module} (i.e., an $R$-module that is a topological space with continuous addition and scalar multiplication) and $N\subset M$ a closed $R$-submodule. For each $m\in M$, let $L_m:M\rightarrow M,~x\mapsto m+x$ be translation by $m$, and {\small $Lt(N,M):=\{L_m|_N:m\in M\}\subset M^N$}. In the anticipated quotient map {\small ~$q:Lt(N,M)\rightarrow q(Lt(N,M))\subset Cl_N(M),~f\mapsto\overline{f(N)}$,~} we have
    {\small\[
    q(Lt(N,M))=\left\{\overline{L_m(N)}=m+N:m\in M\right\}=M/N.
    \]}
\item Let $Y$ be a manifold, {\small $X=\bigsqcup_{y\in Y}X_y=\bigsqcup_{y\in Y}\pi^{-1}(y)$} a fiber bundle over $Y$ with projection $\pi:X\rightarrow Y$. Let {\small $\mathcal{X}:=\{X_y\}_{y\in Y}=\{\pi^{-1}(y)\}_{y\in Y}$}. Then, in the notation of Definition \ref{ProdDescDfn}, we get a map on global sections of $X$ given by
    {\small\[
    \textstyle q:\prod\mathcal{X}\rightarrow Cl_Y(\mathcal{X}),~s\mapsto cl_X(s(Y)).
    \]}In this case, the \emph{limit topology} of Definition \ref{ProdDescDfn} on $X$ plays a nontrivial role as a topology (induced by that of the fibers $X_y$ of the bundle $X$) that can be compared with the underlying topology of $X$.
\end{enumerate}
\end{rmk}

\section{\textnormal{\bf Concrete quotient-realization and metrization of compact-subset hyperspaces}} \label{QRMCSH}
\vspace{0.2cm}
\begin{center}
{\emph{Concrete quotient-realization of hyperspace topologies}}
\end{center}
\vspace{0.1cm}
\noindent Given a space $X$, a set $Y$, and a $q$-full subset $\mathcal{F}\subset X^Y$, we see that $q:(\mathcal{F},q^{-1}(\tau_v))\rightarrow (Cl_Y(X),\tau_v)$ is a surjective open continuous map, hence an open quotient map. In general, $q^{-1}(\tau_v)$ is a highly non-Hausdorff swrc-topology (see Lemma \ref{NonHauLmm}) and therefore not always a convenient swrc-topology, but can be refined/extended to a $\tau_p$-compatible swrc-topology (as shown in Theorem \ref{ScpTopThm}) if $\mathcal{F}$ meets certain conditions. Meanwhile, the existence of a $\tau_p$-compatible (but not necessarily swrc-) topology on $\mathcal{F}$ with a $\tau_v$-compatible quotient is proved in Theorem \ref{ScpTopThmPrv}.

\begin{dfn}[{Chained space, Topologies of uniform/pointwise/compact-uniform convergence}]\label{TopUnConv}
A \textbf{chained-space} $X=(X,\mathcal{U})$ is a space $X$ together with a chain $\mathcal{U}=(\{\mathcal{U}_\lambda\}_{\lambda\in\Lambda},\preceq)$ of open covers of $X$, where ``$\mathcal{U}_\lambda\preceq \mathcal{U}_{\lambda'}$ iff for any $O\in\mathcal{U}_\lambda$ there is $O'\in\mathcal{U}_{\lambda'}$ such that $O\subset O'$ and [for any $O'\in\mathcal{U}_{\lambda'}$ and any $O\in\mathcal{U}_\lambda$, $O'$ is not a proper subset of $O$]''. A collection of open sets $\mathcal{O}$ of $X$ is \textbf{homogeneous} (with respect to $\mathcal{U}$) if $\mathcal{O}\subset\mathcal{U}_{\lambda}$ for some $\lambda$.

Let $X=(X,\mathcal{U})$ be a chained-space, $Y$ a set, and $\mathcal{F}\subset X^Y$. A net $\{f_\alpha\}_{\alpha\in\mathcal{A}}\subset \mathcal{F}$ \textbf{converges uniformly on $Z\subset Y$} to $f\in \mathcal{F}$, written $f_\alpha\stackrel{u|_Z}{\longrightarrow}f$, if for any homogeneous system $\mathcal{O}:=\{O_z\ni f(z):O_z\in\mathcal{U}_\lambda,z\in Z\}$ of open sets of $X$ $\big($e.g., $\mathcal{O}=\{O\}$ for an open neighborhood $\mathcal{U}_\lambda\ni O\supset\overline{f(Z)}$ of $\overline{f(Z)}$ in $X$, or equivalently, $q^{-1}(O^+)\ni f$ of $f$ in $\mathcal{F}\big)$, there exists $\alpha^{\mathcal{O}}\in\mathcal{A}$ such that for each $z\in Z$, $\{f_{\alpha}(z)\}_{\alpha\geq\alpha^{\mathcal{O}}}\subset O_z$ $\big($resp., e.g., $\exists\alpha^O\in\mathcal{A}$ such that $\bigcup_{\alpha\geq\alpha^O}f_\alpha(Z)\subset O\big)$. (\textbf{footnote}\footnote{
We can refer to uniform convergence defined without reference to a chain $\mathcal{U}$ as \textbf{unconditional-uniform convergence}. That is,
a net $\{f_\alpha\}_{\alpha\in\mathcal{A}}\subset \mathcal{F}$ \textbf{converges unconditionally-uniformly on $Z\subset Y$} to $f\in \mathcal{F}$, written $f_\alpha\stackrel{u|_Z}{\longrightarrow}f$, if for any system $\mathcal{O}:=\{O_z\ni f(z):z\in Z\}$ of open sets in $X$, there exists $\alpha^{\mathcal{O}}\in\mathcal{A}$ such that for each $z\in Z$, $\{f_{\alpha}(z)\}_{\alpha\geq\alpha^{\mathcal{O}}}\subset O_z$.\\
\textbf{NB:} All subsequent concepts/results based on \emph{(conditional-) uniform convergence}, that do not refer to (hence do not depend on) any details about the chain $\mathcal{U}$, also apply-to/hold-for \emph{unconditional-uniform convergence}. In particular, Notes \ref{Note1}-\ref{Note3} below (and more) do not depend on $\mathcal{U}$.
}). If a net $\{f_\alpha\}\subset \mathcal{F}$ is uniformly convergent on $Y$, then we simply say $\{f_\alpha\}$ \textbf{converges uniformly} (or that $\{f_\alpha\}$ is a \textbf{uniformly convergent net}), written $f_\alpha\stackrel{u}{\longrightarrow}f$.

\begin{note}\label{Note1} If $Z\subset Z'\subset Y$, then $f_\alpha\stackrel{u|_{Z'}}{\longrightarrow}f$ implies $f_\alpha\stackrel{u|_Z}{\longrightarrow}f$ (which holds because each system of open sets $\mathcal{O}:=\{O_z\ni f(z):O_z\in\mathcal{U}_\lambda,z\in Z\}$ can be extended to a system of open sets $\mathcal{O}':=\{O_z\ni f(z):O_z\in\mathcal{U}_\lambda,z\in Z'\}$). In particular, for any $Z\subset Y$,
\[
\textrm{$f_\alpha\stackrel{u|_Z}{\longrightarrow}f~$ ~implies~ $~f_\alpha|_Z\stackrel{\tau_p}{\longrightarrow}f|_Z~$ (i.e., $~f_\alpha(z)\rightarrow f(z)~$ for all $z\in Z$)}.
\]
\end{note}

A net $\{f_\alpha\}_{\alpha\in\mathcal{A}}\subset\mathcal{F}$ \textbf{converges uniformly with respect to} a family of sets $\mathcal{Z}\subset\mathcal{P}^\ast(Y)$ to $f\in \mathcal{F}$ (making it a \textbf{$\mathcal{Z}$-uniformly convergent net}), written $f_\alpha\stackrel{u|_\mathcal{Z}}{\longrightarrow}f$, if $f_\alpha\stackrel{u|_Z}{\longrightarrow}f$ for all $Z\in\mathcal{Z}$. We note that $\{f_\alpha\}$ is \emph{uniformly convergent} if and only if \emph{$\{Y\}$-uniformly convergent}, if and only if (by Note \ref{Note1} above) uniformly convergent on $Z$ for every $Z\subset Y$.

A subset $\mathcal{C}\subset X^Y$ is \textbf{$\mathcal{Z}$-uniformly closed} if every $\mathcal{Z}$-uniformly convergent net in $\mathcal{C}$ converges to a point in $\mathcal{C}$. The \textbf{topology of $\mathcal{Z}$-uniform convergence} $\tau_{uc|_\mathcal{Z}}$ on $X^Y$ is the topology  whose closed sets are the $\mathcal{Z}$-uniformly closed subsets of $X^Y$. When $\mathcal{Z}=\{Y\}$, we simply drop $\mathcal{Z}$ from the terminology, i.e., ``$\mathcal{Z}$-uniformly closed'', $\tau_{uc|_\mathcal{Z}}$ (``topology of $\mathcal{Z}$-uniform convergence''), etc become ``uniformly closed'', $\tau_{uc}$ (``topology of uniform convergence''), etc. If $\mathcal{Z}=\big\{\{y\}:y\in Y\big\}$, then $\tau_{uc|_\mathcal{Z}}=\tau_p$, in which case we replace ``uniform'' or ``uniformly'' with ``\textbf{pointwise}''. When $Y$ is a space and $\mathcal{Z}=K(Y)$, we call $\tau_{\textrm{cc}}\equiv\tau_{\textrm{cuc}}:=\tau_{uc|_{K(Y)}}$ the \textbf{topology of compact-uniform convergence} (which is a generalization of the usual notion of ``\emph{compact convergence}'' or ``\emph{uniform convergence on compact sets}'' for a metric $X$).

\begin{note}\label{Note2} If $\mathcal{Z}\subset\mathcal{Z}'$, then $\tau_{uc|_\mathcal{Z}}\subset \tau_{uc|_{\mathcal{Z}'}}$. To see this, take a $\tau_{uc|_\mathcal{Z}}$-closed set $\mathcal{C}\subset X^Y$ and show that it is also $\tau_{uc|_{\mathcal{Z}'}}$-closed: Indeed, if $f_\alpha\in\mathcal{C}$, then $f_\alpha\stackrel{\tau_{uc|_{\mathcal{Z}'}}}{\longrightarrow}f$ implies $f_\alpha\stackrel{\tau_{uc|_\mathcal{Z}}}{\longrightarrow}f$, and so $f\in\mathcal{C}$.
\end{note}

\begin{note}\label{Note3} For any $\mathcal{Z}\subset\mathcal{P}(Y)$, with $Y_{\mathcal{Z}}:=\bigcup\{\mathcal{P}(Z):Z\in\mathcal{Z}\}$, we have $~\tau_{uc|_{\mathcal{Z}}}=\tau_{uc|_{\{Y_{\mathcal{Z}}\}}}$. This follows from Note \ref{Note1} above.
\end{note}

\begin{note}\label{Note4} Let $X$ be a metric space, $Y$ a space, and $\mathcal{F}\subset X^Y$. In $\mathcal{F}$, if $f_\alpha\stackrel{\tau_{\textrm{cc}}}{\longrightarrow}f$, then $f_\alpha\stackrel{d^K}{\longrightarrow}f$ (i.e., $\forall\varepsilon>0$, $\exists\alpha^\varepsilon$ s.t. $d^K(f_\alpha,f)<\varepsilon$ $\forall\alpha\geq\alpha^\varepsilon$) for all $K\in K(Y)$, where $d^K(f_\alpha,f):=\sup_{y\in K}d(f_\alpha(y),f(y))$. In particular, in the continuous maps $\mathcal{F}:=C(Y,X)\subset X^Y$, $f_\alpha\stackrel{\tau_{\textrm{cc}}}{\longrightarrow}f$ if and only if $f_\alpha\stackrel{d^K}{\longrightarrow}f$ for every $K\in K(Y)$. (\textbf{footnote}\footnote{
\textbf{Proof:} ($\Rightarrow$): Assume $f_\alpha\stackrel{\tau_{\textrm{cc}}}{\longrightarrow}f$. Fix any $K\in K(Y)$. Then for any collection $\mathcal{O}=\{O_y\ni f(y):y\in K\}$ of open sets in $X$, there is $\alpha^\mathcal{O}$ such that $\{f_\alpha(y)\}_{\alpha\geq\alpha^\mathcal{O}}\subset O_y$ (for all $y\in K$). In particular, for any $\varepsilon>0$, with $\mathcal{O}_\varepsilon:=\{O_{y,\varepsilon}:=B_\varepsilon^K(f(y))~|~y\in K\}$, $\{f_\alpha(y)\}_{\alpha\geq\alpha^{\mathcal{O}_\varepsilon}}\subset O_{y,\varepsilon}~\forall y\in K$ implies
$d^K(f_\alpha,f)=\sup_{y\in K}d(f_\alpha(y),f(y))<\varepsilon~~~~\forall\alpha\geq\alpha^\varepsilon:=\alpha^{\mathcal{O}_\varepsilon}$.

($\Leftarrow$): Assume $f_\alpha\stackrel{d^K}{\longrightarrow}f$ in $C(Y,X)$ for all $K\in K(Y)$. Fix $K\in K(Y)$. For any $\varepsilon>0$, there is $\alpha^\varepsilon$ such that $d^K(f_\alpha,f)<\varepsilon$ for all $\alpha\geq\alpha^\varepsilon$. Therefore, $\{f_\alpha(y)\}_{\alpha\geq\alpha^\varepsilon}\subset B_\varepsilon(f(y))$ for all $\in K$. Consider any collection $\mathcal{O}=\{O_y\ni f(y):y\in K\}$ of open sets in $X$. Let $B_{r_y}(f(y))\subset O_y$ for each $y\in K$. Then by the compactness of $f(K)$, we can choose $r>0$ (independent of $y$) such that $B_r(f(y))\subset O_y$ for each $y\in K$. Hence, with $\varepsilon<r$ and $\alpha^{\mathcal{O}}:=\alpha^\varepsilon$, we get $\{f_\alpha(y)\}_{\alpha\geq\alpha^{\mathcal{O}}}=\{f_\alpha(y)\}_{\alpha\geq\alpha^\varepsilon}\subset  B_\varepsilon(f(y))\subset B_r(f(y))\subset O_y$ for all $y\in K$.
}).
\end{note}

The \textbf{topology of local uniform convergence} $\tau_{luc}$ on $\mathcal{F}\subset X^Y$ is the topology with subbase
\[
\mathcal{S}\mathcal{B}_{luc}:=\left\{B^{d^K}_r(f):=\{g\in\mathcal{F}:d^K(f,g)<r\}~|~r>0,f\in\mathcal{F},K\in K(Y)\right\}.
\]

\begin{note}\label{Note5} By Note \ref{Note4} above: (i) On any $\mathcal{F}\subset X^Y$, ~ $\tau_{\textrm{cc}}\supset\tau_{luc}:=\bigcap\{\tau_{d^K}:K\in K(Y)\}$, ~ where a set $O\subset\mathcal{F}$ is $\tau_{luc}$-open (resp., $C\subset\mathcal{F}$ is $\tau_{luc}$-closed) $\iff$ $d^K$-open (resp., $d^K$-closed) for all $K\in K(Y)$. (ii) On continuous maps $\mathcal{F}:=C(Y,X)\subset X^Y$, $\tau_{\textrm{cc}}=\tau_{luc}$.
\end{note}
\end{dfn}

\begin{dfn}[{An interpretation of the Vietoris topology}]\label{VietTopInt}
Let $X$ be a space and $A_\alpha,A\in\mathcal{P}^\ast(X)$. We say $A_\alpha$ \textbf{centrally-converges} (resp., \textbf{marginally-converges}) to $A$ in $X$, written $A_\alpha\stackrel{+}{\longrightarrow}A$ (resp., $A_\alpha\stackrel{-}{\longrightarrow}A$), if for any open set $O\subset X$ containing $A$ (resp., meeting $A$), there exists $\alpha_+^O$ (resp., $\alpha_-^O$) such that
\[
\textstyle A_\alpha\subset O~\textrm{~}~\forall\alpha\geq\alpha_+^O~\textrm{~}~(\textrm{resp.,~}~A_\alpha\cap O\neq\emptyset~{~}~\forall\alpha\geq\alpha_-^O),
\]
that is, if $A_\alpha\stackrel{\tau_v^+}{\longrightarrow}A$ (resp., $A_\alpha\stackrel{\tau_v^-}{\longrightarrow}A$). A set $\mathcal{C}\subset P^\ast(X)$ is \textbf{centrally-closed} (resp., \textbf{marginally-closed}) if for any net $A_\alpha\in\mathcal{C}$ such that $A_\alpha\stackrel{\tau_v^+}{\longrightarrow}A$ (resp., $A_\alpha\stackrel{\tau_v^-}{\longrightarrow}A$), we have $A\in\mathcal{C}$. The \textbf{upper Vietoris topology} $\tau_v^+$ (resp., \textbf{lower Vietoris topology} $\tau_v^-$) of $\mathcal{P}^\ast(X)$ is the topology whose closed sets are the centrally-closed (resp., marginally-closed) subsets of $\mathcal{P}^\ast(X)$. Noting that $A_\alpha\stackrel{\tau_v}{\longrightarrow}A$ if and only if $A_\alpha\stackrel{\tau_v^+}{\longrightarrow}A$ and $A_\alpha\stackrel{\tau_v^-}{\longrightarrow}A$, the \textbf{Vietoris topology} $\tau_v$ of $\mathcal{P}^\ast(X)$ is the topology whose closed sets are those subsets of $\mathcal{P}^\ast(X)$ that are each both centrally-closed and marginally-closed.
\end{dfn}

\begin{lmm}\label{UnifLimLmm}
Let $X=(X,\mathcal{U})$ be a chained-space, $Y$ a space, $Z\subset Y$, and $f_\alpha\in X^Y$ a net. (i) If $f_\alpha\stackrel{u|_Z}{\longrightarrow}f$, then $f_\alpha(Z)\stackrel{\tau_v^+}{\longrightarrow}f(Z)$.
(ii) If $f_\alpha|_Z\stackrel{\tau_p}{\longrightarrow}f|_Z$, then $f_\alpha(Z)\stackrel{\tau_v^-}{\longrightarrow}f(Z)$ (iff $\overline{f_\alpha(Z)}\stackrel{\tau_v^-}{\longrightarrow}\overline{f(Z)}$). (iii) If $f_\alpha\stackrel{u|_Z}{\longrightarrow}f$, then $f_\alpha(Z)\stackrel{\tau_v}{\longrightarrow}f(Z)$.
\end{lmm}
\begin{proof}
(i) Pick any open neighborhood $O\in\mathcal{U}_\lambda$ of $f(Z)$. Consider the system of open sets $\mathcal{O}:=\{O_y=O\ni f(z)\}_{z\in Z}=\{O\}$. Then there exists $\alpha^{\mathcal{O}}$ such that for each $z\in Z$, $\{f_\alpha(z)\}_{\alpha\geq\alpha^{\mathcal{O}}}\subset O$, which implies $\bigcup_{\alpha\geq\alpha^{\mathcal{O}}}f_\alpha(Z)\subset O$.
(ii) Let $O\in\mathcal{U}_\lambda$ be an open set such that $O\cap f(Z)\neq\emptyset$ (i.e., $f(Z)\in O^-$). Then some $f(z)\in O$. Since $f_\alpha(z)\rightarrow f(z)$, there is $\alpha_z^O$ such that $\{f_\alpha(z)\}_{\alpha\geq\alpha_z^O}\subset O$, i.e., $f_\alpha(Z)\cap O\neq\emptyset$ for all $\alpha\geq\alpha_z^O$. (iii)  By Note \ref{Note1} of Definition \ref{TopUnConv}, $f_\alpha\stackrel{u|_Z}{\longrightarrow}f$ implies $f_\alpha|_Z\stackrel{\tau_p}{\longrightarrow}f|_Z$. So, the conclusion follows by (i) and (ii).
\end{proof}

\begin{thm}[{Construction of an swrc-topology}]\label{ScpTopThm}
Let $X=(X,\mathcal{U})$ be a chained-space, $Y$ a space, and $\mathcal{Z}\subset\mathcal{P}^\ast(Y)$ a cover of $Y$ (i.e., $\bigcup\mathcal{Z}=Y$). Let $\mathcal{F}\subset X^Y$ be a $q$-full (and finitely $q$-stable) subspace such that (i) $\overline{f(Z)}=f(Z)$ for all $\{f\in\mathcal{F},Z\in\mathcal{Z}\}$, and (ii) $(\mathcal{F},\tau)$ is an rc-space with $\tau_{uc|_\mathcal{Z}}\subset\tau$ (for example, we could take $\tau:=\tau_{uc}$). Then the topology $\tau_{\pi}:=\tau_{uc|_\mathcal{Z}}$ on $\mathcal{F}\subset X^Y$ is a $\langle\tau_p\cup q^{-1}(\tau_v)\rangle$-compatible wrc-topology and, moreover, $q:(\mathcal{F},\tau_\pi)\rightarrow (Cl_Y(X),\tau_v)$ is continuous (i.e., $\tau_v\subset\tau_{\pi q}$ in $Cl_Y(X)$ ).

NB: When $Y$ is finite, it suffices (by Theorem \ref{QuotRepVTop}) to replace $\tau_{uc|_\mathcal{Z}}$ with $\tau_p$, in which case, Tychonoff's product theorem ensures that (i) and (ii) are no longer needed.
\end{thm}
\begin{proof}
We need to verify the necessary requirement for $\tau_\pi$ in Definition \ref{StdTopIndSSP}. By Note \ref{Note3} of Definition \ref{TopUnConv} and the equality $\bigcup\mathcal{Z}=Y$, we have $\tau_p\subset \tau_{uc|_\mathcal{Z}}$. Since $\mathcal{F}$ satisfies (i), we also have $q^{-1}(\tau_v)\subset \tau_{uc|_\mathcal{Z}}$ by Lemma \ref{UnifLimLmm}(iii). Therefore $\langle\tau_p\cup q^{-1}(\tau_v)\rangle\subset \tau_{uc|_\mathcal{Z}}$.
\begin{enumerate}[leftmargin=0.7cm]
\item $\tau_{\pi}$ is an rc-topology (hence a wrc-topology) on $\mathcal{F}$: Indeed, for any cr-net $\{f_\alpha\}_{\alpha\in\mathcal{A}}\subset \mathcal{F}$, the rc-space $(\mathcal{F},\tau)$ gives a subnet $f_{\alpha(\beta)}\stackrel{\tau}{\longrightarrow}f\in\mathcal{F}$. So $f_{\alpha(\beta)}\stackrel{\tau_{uc|_\mathcal{Z}}}{\longrightarrow}f\in\mathcal{F}$ (since $\tau_{uc|_\mathcal{Z}}\subset\tau$).

NB: It is clear that if $Y$ is finite, in which case $\tau_{uc|_\mathcal{Z}}\stackrel{(\ast)}{=}\tau_p\stackrel{(\ast)}{=}\langle\tau_p\cup q^{-1}(\tau_v)\rangle=\langle\tau_p\cup q^{-1}(\tau_v^+)\rangle$ (where the equalities ($\ast$) hold because pointwise convergence automatically implies uniform convergence, which in turn implies $q^{-1}(\tau_v)$-convergence by Lemma \ref{UnifLimLmm}(iii)), then the conclusion no longer requires (i) and (ii).

\item $q:(\mathcal{F},\tau_\pi)\rightarrow (Cl_Y(X),\tau_v)$ is continuous (hence continuous on $K\mathcal{F}(Y,X)$), since by construction  $q^{-1}(\tau_v)\subset\tau_{\pi}$ (hence $\tau_v\subset\tau_{\pi q}$).
\end{enumerate}
\end{proof}

\vspace{0.2cm}
\begin{center}
{\emph{Metrization of compact-subset hyperspaces}}
\end{center}
\vspace{0.1cm}
\noindent The existence of a swrc-topology on a $q$-full $\mathcal{F}\subset X^Y$ (say as in Theorem \ref{ScpTopThm}) allows us to concretely establish metrizability (in Theorems \ref{ScpTopLmm0} and \ref{HausdMetriz1}) of indexed compact-subset hyperspaces $K_Y(X)\subset K(X)$ of a metrizable space $X$.

\begin{lmm}[{Compact union I}]\label{VietConnLmm}
Let $X$ be a space and $Y$ a set. If $\mathcal{A}\subset (K_Y(X),\tau_v)$ is compact, then so is $K:=\bigcup_{A\in\mathcal{A}}A\subset X$. Hence, if $\mathcal{B}\subset (K_Y(X),\tau_v)$ is contained in a compact subset of $(K_Y(X),\tau_v)$, then $L:=\bigcup_{B\in\mathcal{B}}B$ is contained in a compact subset of $X$.
\end{lmm}
\begin{proof}
Consider a net $x_\alpha\in K$, and let $x_\alpha\in A_\alpha\in\mathcal{A}$. By the compactness of $\mathcal{A}$, let a subnet $A_{\alpha(\beta)}\stackrel{\tau_v}{\longrightarrow}A_0\in\mathcal{A}$ (i.e., for any $\mathcal{O}:=[\{O_1,...,O_n\}]_v\ni A_0$, some tail $\{A_{\alpha(\beta)\geq\alpha(\beta^\mathcal{O})}\}\subset[\{O_1,...,O_n\}]_v$). We will show there is $x\in A_0$ such that $x_{\alpha(\beta)}\rightarrow x$. On the contrary, suppose that for all $x\in A_0$, $x_{\alpha(\beta)}\not\!\!\longrightarrow x$, i.e., there exists an open set $O_x\ni x$ and a subnet $\{x_{\alpha\circ\beta_x(\gamma)}\}$ that avoids $O_x$. In particular, since $n$ is finite, (a tail of) some subnet $\{x_{\alpha\circ\beta(\gamma)}\}$ avoids $O_1,...,O_n$. But $\{A_{\alpha\circ\beta(\gamma)\geq \alpha\circ\beta(\gamma^\mathcal{O})}\}\subset[\{O_1,...,O_n\}]_v$ implies $\{x_{\alpha\circ\beta(\gamma)\geq \alpha\circ\beta(\gamma^\mathcal{O})}\}\subset\bigcup_iO_i$ (a contradiction). (\textbf{footnote}\footnote{\textbf{Alternative proof}: Let $\{O_i:i\in I\}$ be an open cover of $K$ (hence also an open cover of each $A\in\mathcal{A}$) in $X$. For each $A\in \mathcal{A}$, let $\{O^A_{i_j}:j\in I_A\}\subset \{O_i:i\in I\}$ be a finite subcover of $A$ (where wlog $A\cap O_{i_j}^A\neq\emptyset$ for all $j$). Then $A\in\langle O^A_{i_j}:j\in I_A \rangle_v$, and so $\big\{[\{O^A_{i_j}:j\in I_A\}]_v\big\}_{A\in\mathcal{A}}$ is a $\tau_v$-open cover of $\mathcal{A}$, which therefore has a finite $\tau_v$-subscover $\big\{[\{O^A_{i_j}:j\in I_A\}]_v\big\}_{A\in F}$, for some finite set $F\subset\mathcal{A}$. Hence $\{O^A_{i_j}:j\in I_A,A\in F\}$ is a finite subcover of $K$ in $X$.}).
\end{proof}

When $X$ is a metric space (in which case $(K(X),\tau_v)\cong(K(X),\tau_{d_H})$ by Lemma \ref{VietTopMetr}), the above result becomes equivalent to the following.

\begin{lmm}[{Compact union II: \cite[Lemma 3.5(ii)]{akofor2024}}]\label{ConvCompLmm}
Let $X=(X,d)$ be a metric space and $Y$ a set. If $\mathcal{C} \subset \big(K_Y(X),d_H\big)$ is compact, then $K:=\bigcup_{C\in\mathcal{C}}C\subset X$ is compact.
\end{lmm}

We review the proof of the following well known result.

\begin{lmm}[{\cite[Theorem 3.1]{IllanNadl1999}}]\label{VietTopMetr}
If $X$ is a metrizable space, then so is $(K(X),\tau_v)$. Moreover, if $X\cong(X,d)$, then $(K(X),\tau_v)\cong(K(X),d_H)$. Conversely, if $X$ is $T_1$ (i.e., all singletons of $X$ are closed) and $K(X)$ is metrizable, then $X$ is metrizable (as a subspace of a metrizable space).
\end{lmm}
\begin{proof}
Let $X=(X,d)$. (i) \textbf{Showing $\tau_v\subset\tau_{d_H}$ in $K(X)$}: Let $O\subset X$ be open. We need to show $O^+$ and $O^-$ are in $\tau_{d_H}$. First, let $A_n\in (O^+)^c=\{K\in K(X):K\cap O^c\neq\emptyset\}=(O^c)^-$ such that $A_n\stackrel{d_H}{\longrightarrow}A$. Then with $a_n\in A_n\cap O^c\subset A\cup\bigcup_nA_n$ (a compact set), let $d(a_{f(n)},e)\rightarrow 0$, where $e\in O^c$ (a closed set). Then
\begin{align*}
d(e,A)&=\inf_{a\in A}d(e,a)\leq \inf_{a\in A}[d(e,a_{f(n)})+d(a_{f(n)},a)]=d(e,a_{f(n)})+d(a_{f(n)},A)\\
&\leq d(e,a_{f(n)})+d_H(A_{f(n)},A)\rightarrow 0,\\
&~~\Rightarrow~~e\in A~~\Rightarrow~~O^c\cap A\neq\emptyset,~~\Rightarrow~~A\in (O^+)^c,
\end{align*}
and so $(O^+)^c$ is $d_H$-closed, i.e., $O^+$ is $d_H$-open.

Second, let $A_n\in (O^-)^c=\{K\in K(X):K\subset O^c\}=(O^c)^+$ such that $A_n\stackrel{d_H}{\longrightarrow}A$. Fix $a\in A$. Then, for some $a_n=a_n(a)\in A_n$, we have $d(a,a_n)=d(a,A_n)\leq d_H(A,A_n)\rightarrow 0$. Since $a_n\in A\cup\bigcup_nA_n$ (a compact set), let $d(a_n,e)\rightarrow 0$ for some $e\in O^c$ (since $A_n\subset O^c$ and $O^c$ is a closed set). Then
\[
d(a,e)\leq d(a,a_n)+d(a_n,e)\rightarrow 0,~~\Rightarrow~~a=e,~~\Rightarrow~~A\subset O^c,
\]
and so $(O^-)^c$ is $d_H$-closed (i.e., $O^-$ is $d_H$-open).

(ii) \textbf{Showing $\tau_{d_H}\subset\tau_v$ in $K(X)$}: Let $\mathcal{C}\subset K(X)$ be a $d_H$-closed set. Let $A_\alpha\in\mathcal{C}$ such that $A_\alpha\stackrel{\tau_v}{\longrightarrow}A$. We need to show $A\in\mathcal{C}$. Suppose $A\in\mathcal{C}^c$. For any $\mathcal{O}:=[O_1,...,O_n]_v\ni A$, some tail $\{A_{\alpha\geq\alpha^\mathcal{O}}\}\subset [O_1,...,O_n]_v$. With $\varepsilon^\mathcal{O}:=\max_i\textrm{diam}(O_i)$, we get ~$\{A_{\alpha\geq\alpha^\mathcal{O}}\}\subset \overline{B_{\varepsilon^\mathcal{O}}^{d_H}(A)}$,~ since for $\alpha\geq\alpha^\mathcal{O}$,
\begin{align*}
\textstyle d_H(A,A_\alpha)&=\max_{A\leftrightarrow A_\alpha}\max_{a\in A}d(a,A_\alpha)\leq\max_{A\leftrightarrow A_\alpha}\max_{a\in A}d(a,A_\alpha\cap O_{i_a})|_{a\in O_{i_a}}\leq\varepsilon^\mathcal{O}.
\end{align*}
Since $A\in\mathcal{C}^c$ (a $d_H$-open set), some $B_r^{d_H}(A)\subset\mathcal{C}^c$. So, by choosing $\mathcal{O}$ such that $\varepsilon^\mathcal{O}<r$ (which is possible by the compactness of $A$), we get $\{A_{\alpha\geq\alpha^\mathcal{O}}\}\subset \mathcal{C}^c$ (a contradiction).
\end{proof}

\begin{thm}\label{ScpTopLmm0}
If $X=(X,d)$ is a metric space and $Y$ a finite set, then ~$~(Cl_Y(X),\tau_{pq})=(Cl_Y(X),\tau_v)=(Cl_Y(X),d_H)$.
\end{thm}
\begin{proof}
Since $Cl_Y(X)=FS_{|Y|}(X)\subset K(X)$, it follows from Lemma \ref{VietTopMetr} that ``$(Cl_Y(X),\tau_v)=(Cl_Y(X),d_H)$'', where by Theorem \ref{QuotRepVTop}, $\tau_{pq}=\tau_v$ in $Cl_Y(X)$.
\end{proof}

Theorem \ref{ScpTopLmm0} is metrization of the quotient of the swrc-topology $(X^Y,\tau_p)$ due to a finite $Y$. Metrization of the quotient of a general $q$-full swrc-space $(\mathcal{F},\tau_\pi)\subset X^Y$ is given by Theorem \ref{HausdMetriz1}.

\begin{lmm}\label{NonHauLmm}
Let $X$ be a set, $(Y,\tau)$ a space, $f:X\rightarrow Y$ a map, and $y_\alpha\stackrel{\tau}{\longrightarrow}y$ a convergent net in $Y$. If $x_\alpha\in f^{-1}(y_\alpha)$ and $x\in f^{-1}(y)$, then ~$x_\alpha\stackrel{f^{-1}(\tau)}{\longrightarrow}x$ in $X$ (where $f^{-1}(\tau):=\{f^{-1}(U):U\in\tau\}$). (\textbf{footnote}\footnote{
Since $f:(X,f^{-1}(\tau))\rightarrow(Y,\tau)$ is continuous, for every convergent net $x_\alpha\stackrel{f^{-1}(\tau)}{\longrightarrow}x$ in $X$, we have $f(x_\alpha)\stackrel{\tau}{\longrightarrow}f(x)$ in $Y$.
}).
\end{lmm}
\begin{proof}
Fix any $x_\alpha\in f^{-1}(y_\alpha)$ and any $x\in f^{-1}(y)$. Let $O\ni x$ be an $f^{-1}(\tau)$-open neighborhood of $x$ in $X$, i.e., $x\in O=f^{-1}(U)$ for an open set $U\subset Y$. Then $y=f(x)\in U$ and so some tail $\{y_\alpha\}_{\alpha\geq\alpha^U}\subset U$. Therefore, $\{x_\alpha\}_{\alpha\geq\alpha^U}\subset f^{-1}(\{y_\alpha\}_{\alpha\geq\alpha^U})\subset f^{-1}(U)=O$.
\end{proof}

\begin{thm}\label{QuotVietCont2}
Let $X=(X,d)$ be a metric space, $Y$ a set, and $\mathcal{F}\subset X^Y$ a $q$-full subset. Then the space {\small $\big(K\mathcal{F}(Y,X),q^{-1}(\tau_v)\big)$} of precompact image maps $K\mathcal{F}(Y,X):=\mathcal{F}\cap q^{-1}(K_Y(X))$ in $\mathcal{F}$ is pseudometrized by
\[
d^q_H(f,g):=d_H\left(q(f),q(g)\right).
\]
That is, $(K\mathcal{F}(Y,X),q^{-1}(\tau_v))\cong \big(K\mathcal{F}(Y,X),q^{-1}(\tau_{d_H})\big)$.
\end{thm}
\begin{proof}
This is precisely the proof of Lemma \ref{VietTopMetr}, with $K(X)$ replaced by $K_Y(X)$, along with the following basic observations (for an open set $O\subset X$): (1) $A=q(f)\in O^{\pm}$ $\iff$ $f\in q^{-1}(O^{\pm})$ in $\mathcal{F}$. (2) $B=q(g)\in B_r^{d_H}\left(q(f)\right)$ $\iff$ $g\in B_r^{d^q_H}(f)$ in $\mathcal{F}$.
\end{proof}

\begin{rmk}
In Theorem \ref{QuotVietCont2}, if the \emph{continuous} maps $C(Y,X)$ in particular satisfy $q(C(Y,X))=q(X^Y)$, then with ~$\mathcal{F}:=C(Y,X)$, the precompact image maps ~$K\mathcal{F}(Y,X):=\mathcal{F}\cap q^{-1}(K_Y(X))$~ can be replaced with the precompact image \emph{continuous} maps ~$KC(Y,X):=C(Y,X)\cap q^{-1}(K_Y(X))$.
\end{rmk}

\begin{thm}[{Metrization of compact-subset hyperspaces}]\label{HausdMetriz1}
Let $X=(X,d)$ be a metric space, $Y$ a set, and $(\mathcal{F},\tau_\pi)\subset X^Y$ a $q$-full swrc-space. Then $(K_Y(X),\tau_{\pi q})\cong(K_Y(X),d_H)\cong(K_Y(X),\tau_v)$.
\end{thm}
\begin{proof}
By the continuity of the map
\begin{align}
\label{QuMapCntEq2} q:(K\mathcal{F}(Y,X),\tau_{\pi})\rightarrow(K_Y(X),\tau_v)\stackrel{\textrm{Lemma \ref{VietTopMetr}}}{=}(K_Y(X),d_H),
\end{align}
the $\tau_{\pi q}$-topology of $K_Y(X)$ contains the $d_H$-topology (i.e., $\tau_{d_H}\subset\tau_{\pi q}$). This shows, among other things, that $(K_Y(X),\tau_{\pi q})$ is Hausdorff, since $\tau_{d_H}$ is Hausdorff. In $K_Y(X)$, we also have $\tau_{\pi q}\subset \tau_{d_H}$, as shown next.

Let $\mathcal{C}\subset K_Y(X)$ be a $\tau_{\pi q}$-closed set. Consider a sequence $\left\{q(f_n)\right\}_n\subset\mathcal{C}$ such that $q(f_n)\stackrel{d_H}{\longrightarrow}q(f)$ for some $q(f)\in BCl_Y(X)$. We want to show that $q(f)\in\mathcal{C}$ (in which case $\mathcal{C}$ is also $d_H$-closed). In $\mathcal{F}$, let $g_n\in q^{-1}\left(q(f_n)\right)\subset q^{-1}(\mathcal{C})$, i.e., $q(g_n)=q(f_n)$ (\textbf{footnote}\footnote{A natural choice here is $g_n:=f_n$. So, when $Y$ is finite, it is enough to take $\tau_\pi:=\tau_p$ and thereby automatically get a pointwise convergent subnet of the $cr$-sequence $\{g_n\}$ by Tychonoff's product theorem (thereby obtaining an alternative proof of Theorem \ref{ScpTopLmm0} without Theorem \ref{QuotRepVTop}).}).

Since $\left\{q(f)\right\}\cup\left\{q(f_n)\right\}_n\subset \big(K_Y(X),d_H\big)$ is compact, $K=q(f)\cup\bigcup_nq(f_n)\subset (X,d)$ is compact by Lemma \ref{ConvCompLmm}. So, by the definition of $\tau_{\pi}$, the cr-sequence $\{g_n:Y\rightarrow q(f_n)\subset K\subset(X,d)\}_n$ has a $\tau_{\pi}$-convergent subnet $\{g_{n(\alpha)}\}\stackrel{\tau_{\pi}}{\longrightarrow} h$, where $h\in q^{-1}(\mathcal{C})$ since $q^{-1}(\mathcal{C})$ is $\tau_{\pi}$-closed. By the continuity of (\ref{QuMapCntEq2}),
\[
q(g_{n(\alpha)})=q(f_{n(\alpha)})\stackrel{d_H}{\longrightarrow}q(f)=q(h)\in\mathcal{C}.
\]
\end{proof}

\section{\textnormal{\bf Quotient-realization of hyperspace topologies}} \label{QLHT}
\noindent When $Y$ is not finite, it is not clear whether or not $\mathcal{F}\subset X^Y$ always admits a swrc-topology, even though we have seen in Theorem \ref{ScpTopThm} that a swrc-topology can exist under special conditions (which allow us to obtain a swrc-topology from \emph{a $\tau_p$-compatible topology with a $\tau_v$-compatible quotient}). We will prove the existence of \emph{a $\tau_p$-compatible topology with $\tau_v$ as its quotient} in Theorem \ref{ScpTopThmPrv}.

In this section, unless said otherwise, let $X$ be a space, $Y$ a set, and $\mathcal{F}\subset X^Y$ a $q$-full subset.

\begin{dfn}[{$q$-lifts of subset hyperspace topologies}]
Let $\tau$ be a topology on $Cl(X)$. A \textbf{$q$-lift} of $\tau$ is any topology $\widetilde{\tau}$ on $\mathcal{F}$ such that $\widetilde{\tau}_q=\tau$ in $Cl_Y(X)$. We show in Theorem \ref{ScpTopThmPrv} that if $\tau$ is $\tau_v$-compatible, then a $\tau_p$-compatible $q$-lift $\widetilde{\tau}$ of $\tau$ can be chosen and the choice is maximal with respect to subsets of a $q$-full finitely $q$-stable $\mathcal{F}\subset X^Y$.
\end{dfn}

\begin{rmk}[{Existence of $q$-lifts, Smallest $q$-lift}]\label{TauVEqQMap}
Observe that given any topology $\tau_0$ on $Cl(X)$, the topology $q^{-1}(\tau_0):=\{q^{-1}(B):B\in\tau_0\}$ on $\mathcal{F}$ is the \textbf{smallest $q$-lift} of $\tau_0$ (where we note that $q:\big(\mathcal{F},q^{-1}(\tau_0)\big)\rightarrow (Cl_Y(X),\tau_0)$ is an open continuous map, hence an open quotient map). That is, $q^{-1}(\tau_0)\subset\widetilde{\tau}_0$ for every $q$-lift $\widetilde{\tau}_0$ of $\tau_0$. So, ~$q^{-1}(\tau_0)=\bigcap\{\textrm{$q$-lifts $\widetilde{\tau}_0$ of $\tau_0$}\}$.
\end{rmk}

\begin{rmk}[{Temporary notation for the proof of Theorem \ref{ScpTopThmPrv}}]
Let $X$ be a space, $Y$ a set, $\mathcal{I},\mathcal{J}\subset X^Y$ subsets, $\tau_i$ any topology on $\mathcal{I}$, and $\tau_j$ any topology on $\mathcal{J}$. Consider the map $q:X^Y\rightarrow Cl_Y(X)$.
\begin{enumerate}
\item  As usual, in any space $(X^Y,\tau)$, with $\tau\cap\mathcal{I}:=\{A\cap\mathcal{I}:A\in\tau\}$, we write the subspace $(\mathcal{I},\tau\cap\mathcal{I})\subset(X^Y,\tau)$ simply as $(\mathcal{I},\tau)\subset(X^Y,\tau)$ or as $\mathcal{I}\subset(X^Y,\tau)$.
\item In the restriction $q:(\mathcal{I},\tau_i)\subset X^Y\rightarrow (q(\mathcal{I}),\tau_{iq})\subset Cl_Y(X)$, if $B\subset Cl_Y(X)$, then we write ``$B\in\tau_{iq}$'' to mean ``$B\cap q(\mathcal{I})\in\tau_{iq}$''.
\item In the restrictions $q:(\mathcal{I},\tau_i)\rightarrow (q(\mathcal{I}),\tau_{iq})$ and $q:(\mathcal{J},\tau_j)\rightarrow (q(\mathcal{J}),\tau_{jq})$, we write ``$\tau_{iq}\subset\tau_{jq}$'' to mean ``$\forall B\subset Cl_Y(X)$, $B\in\tau_{iq}$ $\Rightarrow$ $B\in\tau_{jq}$ (i.e., $B\cap q(\mathcal{I})\in\tau_{iq}$ $\Rightarrow$ $B\cap q(\mathcal{J})\in\tau_{jq}$)''.
\end{enumerate}
\end{rmk}

\begin{thm}[{Existence of $\tau_p$-compatible $q$-lifts}]\label{ScpTopThmPrv}
Let $X$ be a space, $Y$ a set, $\mathcal{F}\subset X^Y$ a $q$-full finitely $q$-stable subset, and $\tau_0\supset\tau_v$ a topology on $Cl(X)$. Further suppose that some finite-subset hyperspace $FS_n(X)\subset (Cl_Y(X),\tau_v)$ is closed. Then there exists a topology $\widetilde{\tau}_0\supset\tau_p$ on $\mathcal{F}$, which is maximal with respect to subsets of $\mathcal{F}$, such that $\tau_0\supset\widetilde{\tau}_{0q}\supset\tau_v$ in $Cl_Y(X)$.
\end{thm}
\begin{proof}
Consider the set $\mathcal{P}:=\big\{(\mathcal{A},\tau)$: (i) $\mathcal{A}\subset \mathcal{F}$, (ii) $\tau_p\subset\tau$, (iii) $\tau_v\subset\tau_q\subset\tau_0$ in $q(\mathcal{A})$ (\textbf{footnote}\footnote{
Here, ``$\tau_v\subset\tau_q$'' (with respect to the quotient map $q:(\mathcal{A},\tau)\rightarrow (q(\mathcal{A}),\tau_q)\subset Cl_Y(X)$ ) really means ``$\tau_v\cap q(\mathcal{A})\subset\tau_q$'', where $\tau_v\cap q(\mathcal{A}):=\{B\cap q(\mathcal{A}):B\in\tau_v\}$ is the actual subspace topology in the conventional subspace ``$(q(\mathcal{A}),\tau_v)\subset (Cl_Y(X),\tau_v)$''. Note that we may also express \textbf{``$q:(\mathcal{A},\tau)\rightarrow (Cl_Y(X),\tau_v)$~is continuous''} more explicitly as \textbf{``$q:(\mathcal{A},\tau)\rightarrow (q(\mathcal{A}),\tau_q,\tau_v)\subset(Cl_Y(X),\tau_v)$~is continuous with respect to $\tau_v$''}.
})$\big\}$ as a poset with respect to ``$(\mathcal{A}_1,\tau_1)\leq(\mathcal{A}_2,\tau_2)$ if $\mathcal{A}_1\subset \mathcal{A}_2$ and $\tau_1\subset\tau_2$ (in the sense that $\tau_1=\tau_2\cap\mathcal{A}_1:=\big\{A\cap \mathcal{A}_1:A\in\tau_2\big\}$)''. We note that if $\tau_1\subset\tau_2$, then $\tau_{1q}\supset\tau_{2q}$, which holds because, given $B\in Cl_Y(X)$, $B\in\tau_{2q}$ (i.e., $B\cap q(\mathcal{A}_2)\in\tau_{2q}$) $\iff$ $q^{-1}(B)\cap \mathcal{A}_2\in\tau_2$ $~\stackrel{\textrm{ordering}}{\Longrightarrow}~$ $q^{-1}(B)\cap \mathcal{A}_1=(q^{-1}(B)\cap \mathcal{A}_2)\cap \mathcal{A}_1\in\tau_1$ $\iff$ $B\in\tau_{1q}$ (i.e., $B\cap q(\mathcal{A}_1)\in\tau_{1q}$).

By Theorem \ref{QuotRepVTop}(\ref{VietQRep1}) and Theorem \ref{QuotRepVTop}'s Remark (\ref{SSpConnect0}), $\mathcal{P}\neq\emptyset$ (because with $1\leq n\leq|Y|$, $\mathcal{A}:=\{f\in \mathcal{F}:|q(f)|\leq n\}=q^{-1}(FS_n(X))\cap\mathcal{F}$, and $\tau:=\tau_p$, we get $\tau_q=\tau_v$ in $q:(\mathcal{A},\tau)\rightarrow(q(\mathcal{A}),\tau_q)\subset FS_n(X)$). Given a chain $\{(\mathcal{A}_\lambda,\tau_\lambda)\}_{\lambda\in\Lambda}$ in $\mathcal{P}$, let $\mathcal{A}:=\bigcup_{\lambda}\mathcal{A}_\lambda$ and $\tau:=\bigcup_{\lambda}\tau_\lambda$ (i.e., for any $A\subset\mathcal{A}$, we have $A\in\tau$ $\iff$ $A\cap \mathcal{A}_\lambda\in\tau_\lambda$ for each $\lambda$). Then (i) $\mathcal{A}\subset \mathcal{F}$, (ii) $\tau_p\subset\tau$, and (iii) $\tau_v\subset\tau_q\subset\tau_0$ (since $\tau_v\subset\tau_{\lambda q}\subset\tau_0$ for each $\lambda$ and $\tau_q=\bigcap_\lambda\tau_{\lambda q}$). Here, $\tau_q=\bigcap_\lambda\tau_{\lambda q}$ is due to the following: Given $B\in Cl_Y(X)$, for each $\lambda$, $B\in\tau_{\lambda q}$ $\stackrel{\textrm{def.}}{\iff}$ $q^{-1}(B)\cap \mathcal{A}_\lambda\in\tau_\lambda$, and so $B\in\tau_q$ $\iff$ $q^{-1}(B)\in\tau$ $\stackrel{\textrm{def.}}{\iff}$ $q^{-1}(B)\cap \mathcal{A}_\lambda\in\tau_\lambda$ $\forall\lambda$ $\iff$ $B\in\tau_{\lambda q}$ $\forall\lambda$.

Therefore $(\mathcal{A},\tau)$ is an upper bound of the given chain in $\mathcal{P}$. By Zorn's lemma, $\mathcal{P}$ has a maximal element $(\mathcal{A}',\tau')$. Suppose $\mathcal{A}'\neq \mathcal{F}$. Let $f\in \mathcal{F}\backslash \mathcal{A}'$, $\mathcal{A}'':=\mathcal{A}'\cup\{f\}$, and $\tau''$ be the topology on $\mathcal{A}''$ given by $\tau'':=\{A\subset \mathcal{A}'':A\cap \mathcal{A}'\in\tau'\}=\tau'\cup\big\{A\cup\{f\}:A\in\tau'\big\}$ (\textbf{footnote}\footnote{This is the superspace topology defined in the footnote(s) on page \pageref{SupSsFtNt}.}).

We \emph{claim} that
$(\mathcal{A}',\tau')<\big(\mathcal{A}'',\tau''\big)\in\mathcal{P}$. \emph{Proof of claim:} It is clear that $\mathcal{A}'\subsetneq \mathcal{A}''$ and $\tau'\subset\tau''$ (which implies $\tau'_q\supset\tau_q''$ as before), and so $(\mathcal{A}',\tau')<\big(\mathcal{A}'',\tau''\big)$. Next, $\big(\mathcal{A}'',\tau''\big)\in\mathcal{P}$ follows from (i) $\mathcal{A}''\subset \mathcal{F}$, (ii) $\tau_p\subset\tau''$, and (iii) $\tau_v\subset\tau_q'=\tau_q''\subset\tau_0$, since we also have $\tau_q'\subset\tau_q''$ (because given $B\in Cl_Y(X)$, $B\in\tau_q'$ $\iff$ $q^{-1}(B)\cap \mathcal{A}'=(q^{-1}(B)\cap\mathcal{A}'')\cap \mathcal{A}'\in\tau'$ $\stackrel{\textrm{construction of $\tau''$}}{\Longrightarrow}$ $q^{-1}(B)\cap\mathcal{A}''\in\tau''$ $\iff$ $B\in\tau_q''$). This completes the proof of the claim.

But $(\mathcal{A}',\tau')<\big(\mathcal{A}'',\tau''\big)\in\mathcal{P}$ contradicts maximality of $(\mathcal{A}',\tau')$ in $\mathcal{P}$. So, we can set $\widetilde{\tau}_0:=\tau'$.
\end{proof}

\begin{question}\label{ScpTopQuest}
In Theorem \ref{ScpTopThmPrv}, when can we choose the topology $\widetilde{\tau}_0$ to be a rc-, wrc-, or swrc-topology?
\end{question}

\section{\textnormal{\bf Conclusion and questions}}\label{DiscQuest}
\noindent For a space $X$, we have seen (Theorem \ref{QuotRepVTop}) that certain hyperspaces $\mathcal{J}\subset Cl(X)$ of $X$ can be described as quotients of function spaces $\mathcal{F}\subset X^Y$ in a natural way. Following this, we have discussed the concrete realization of certain preferred function space topologies (Theorem \ref{ScpTopThm}), the metrization of compact-subset hyperspaces (Theorems \ref{ScpTopLmm0} and \ref{HausdMetriz1}), and the existence of $\tau_p$-compatible function space topologies with $\tau_v$-compatible quotients  (Theorem \ref{ScpTopThmPrv}).

In addition to Questions \ref{StdTopQsn} and \ref{ScpTopQuest}, we have the following interesting questions.

\begin{question}
Let $X,Y$ be spaces, $\mathcal{F}\subset X^Y$ a $q$-full subset, $Z\subset Y$, and $Cl_Z(X)\subset (Cl_Y(X),\tau)$ a $\tau$-closed subset, in which case a quotient map $q:(\mathcal{F},\widetilde{\tau})\rightarrow (Cl_Y(X),\tau)$ restricts to a quotient map $q:(\mathcal{F}_Z,\widetilde{\tau})\subset X^Y\rightarrow (Cl_Z(X),\tau)$, where $\mathcal{F}_Z:=q^{-1}(Cl_Z(X))\cap\mathcal{F}$. If $X$ is a metric space, can we find a Lipschitz retraction
\[
(Cl_Y(X),\tau)\rightarrow (Cl_Z(X),\tau)?
\]
This question is of interest especially in the case where $Y$ is finite (see \cite{akofor2019,akofor2020,AkoKov2021,BacKov2016,COR2021,Kovalev2016} and references therein).
\end{question}

\begin{question}\label{GHDQuestion01}
Let $X$ be a space and $Y$ a set. From Definition \ref{ProdDescDfn}, for which $Y$ (of smallest possible cardinality) does the equality $Cl_Y(X)=Cl(X)$ hold? Also, what is the cardinally-smallest $q$-full subset $\mathcal{F}\subset X^Y$?
\end{question}

\begin{question}[{Path representation problem}]\label{GHDQuestion00}
Let $X$ be a space, $Y$ a set, $\mathcal{F}\subset X^Y$ a $q$-full subset, and $\tau_\pi$ a swrc-topology on $\mathcal{F}$. It is clear that a path $\eta:[0,1]\rightarrow (\mathcal{F},\tau_\pi)$ gives a path $q\circ\eta:[0,1]\stackrel{\eta}{\longrightarrow}\mathcal{F}\stackrel{q}{\longrightarrow}(Cl_Y(X),\tau_{\pi q})$, since the composition of continuous maps is continuous. Conversely, (when) does every path $\gamma:[0,1]\rightarrow (Cl_Y(X),\tau_{\pi q})$ come from (or lift as $\gamma=q\circ\eta$ to) a path $\eta:[0,1]\rightarrow (\mathcal{F},\tau_\pi)$?

Whenever the answer to this question is positive, every path $\gamma:[0,1]\rightarrow (Cl_Y(X),\tau_{\pi q})$ is expressible in the form
\begin{align*}
\gamma(t)&=q\circ\eta(t)=q\big(\eta(t)(Y)\big)=cl_X\eta(t)(Y)= cl_X\{\eta(t)(y):y\in Y\}\\
&\equiv cl_X{\{\gamma_r(t):r\in \Gamma\}},~~\textrm{~}\forall t\in[0,1],
\end{align*}
for a set of paths $\{\gamma_r:[0,1]\rightarrow X\}_{r\in \Gamma}$ in $X$. Of course, by the axiom of choice, any path $\gamma:[0,1]\rightarrow (Cl_Y(X),\tau_{\pi q})$ can be written as
\[
\gamma=q\circ\eta:[0,1]\stackrel{\eta}{\longrightarrow}\mathcal{F}\stackrel{q}{\longrightarrow}(Cl_Y(X),\tau_{\pi q}),~t\mapsto q\circ\eta(t)=cl_X\big(\eta(t)(Y)\big)
\]
for a (not necessarily continuous) selection
\[
\eta:[0,1]\rightarrow (\mathcal{F},\tau_\pi),~\textrm{~}~t\mapsto \eta(t)~\in~ q^{-1}(\gamma(t)).
\]
This implies every path $\gamma:[0,1]\rightarrow (Cl_Y(X),\tau_{\pi q})$ is naturally expressible in the form
\[
\gamma(t)=cl_X{\{\gamma_r(t):r\in \Gamma\}}\equiv cl_X\{\gamma_y(t):=\eta(t)(y)~|~y\in Y\},~\textrm{~}~\forall t\in[0,1],
\]
for (not necessarily continuous) maps $\{\gamma_r:[0,1]\rightarrow X\}_{r\in \Gamma}$ (which therefore need not be paths in $X$).

By the above paragraph, Question \ref{GHDQuestion00} is relevant to the ``\emph{path representation}'' problem considered in \cite{akofor2024} and therefore relevant to \cite[Question 5.1]{akofor2024}, even though the current topology $\tau_{\pi q}$ on $BCl_Y(X)\subset Cl_Y(X)$ is in general not the same as the $d_H$-topology. Since $\gamma^{-1}(\mathcal{O})=\eta^{-1}(q^{-1}(\mathcal{O}))$ for any open set $\mathcal{O}\subset (Cl_Y(X),\tau_{\pi q})$, if $\tau_\pi=q^{-1}(\tau_{\pi q})$, then $\gamma$ is continuous iff $\eta$ is continuous (but in general, if we have proper containment $q^{-1}(\tau_{\pi q})\subsetneq \tau_\pi$, then $\gamma=q\circ\eta$ can be continuous even when $\eta$ is not continuous). Therefore, for $\gamma=q\circ\eta$ to be continuous, we only need a $q^{-1}(\tau_{\pi q})$-continuous (not necessarily a $\tau_\pi$-continuous) selection
\[
\eta:[0,1]\rightarrow(\mathcal{F},q^{-1}(\tau_{\pi q})),~\textrm{~}~t\mapsto\eta(t)~\in~ q^{-1}(\gamma(t)).
\]
\end{question}
In particular, if $Y$ is finite and $X^Y$ is Hausdorff, then by \cite[Theorem 3.3]{Vesely1991}, the continuous selection $\eta$ always exists, in which case, every path $\gamma:[0,1]\rightarrow Cl_Y(X)=FS_{|Y|}(X)$ has the form $\gamma(t)=\{\gamma_r(t):r\in\Gamma\}$, for paths $\gamma_r:[0,1]\rightarrow X$.

\section*{\textnormal{\bf Acknowledgments}}

This manuscript has been improved using feedback from referees and editors of The Houston Journal of Mathematics.


\vspace{0.2cm}
\hrule
\end{document}